\documentclass[11pt,letterpaper]{amsart}

\usepackage{fullpage,url,amssymb}
\usepackage[pdfencoding=auto,psdextra,colorlinks=true, citecolor=blue,bookmarksdepth=10]{hyperref}
\usepackage{bookmark}
    \usepackage{amssymb,amsrefs}
\usepackage{amsmath}
\usepackage{amscd,latexsym}
\usepackage{bookmark}
 \usepackage{amssymb,amsfonts,bm}
\usepackage[all,arc]{xy}
\usepackage{enumerate}
\usepackage{mathrsfs}
\usepackage{amscd}
\usepackage{tikz}
\usepackage{tikz-cd}
\usepackage{amsmath}
\usepackage{latexsym}
\usepackage{mathdots}
\usepackage{amsrefs}
\usepackage{graphicx}
\usepackage{mathtools}
\usepackage{leftidx}
\usepackage{tensor}
\usepackage{dsfont}
\usepackage{tikz-cd} 
\usepackage[new]{old-arrows}
\usepackage{xcolor}
\usepackage{bbm}
\usepackage{enumitem}

\theoremstyle{plane}
\newtheorem{theorem}{Theorem}[section]
\newtheorem{lemma}[theorem]{Lemma}
 \newtheorem{corollary}[theorem]{Corollary}
      \newtheorem{proposition}[theorem]{Proposition}

        \newtheorem{theorem*}{Theorem}

\theoremstyle{definition}
\newtheorem*{definition*}{Definition}
\newtheorem{definition}[theorem]{Definition}

\theoremstyle{remark}

\newtheorem*{acknowledgments}{Acknowledgments}

\numberwithin{equation}{section}

\begin{document}

\title{Rankin-Selberg integrals for local symmetric square factors on $GL\mathrm{(2)}$}
\author{Yeongseong Jo}
\address{Department of Mathematics, The University of Iowa, Iowa City, IA 52242, USA}
\email{jo.59@buckeyemail.osu.edu}








\begin{abstract}
Let $\pi$ be an irreducible admissible (complex) representation of $GL(2)$ over a non-archimedean characteristic zero local field with odd residual characteristic. In this paper we prove the equality between the local symmetric square $L$-function associated to $\pi$ arising from integral representations and the corresponding Artin $L$-function for its Langlands parameter through the local Langlands correspondence. With this in hand, we show the stability of local symmetric $\gamma$-factors attached to $\pi$ under highly ramified twists.
\end{abstract}

\maketitle



\section{Introduction}

Let $F$ be a $p$-adic field with $p \neq 2$. We study the Rankin-Selberg integral of the local symmetric square $L$-functions for an irreducible admissible (complex) representation $\pi$ of $GL_2(F)$, introduced by Yamana \cite{Ya17}. The ultimate goal of this paper is twofold. The first aim is to show that the automorphic symmetric square $L$-function attached to $\pi$ by the theory of integral representations is equal to the corresponding Artin $L$-function for its Langlands parameter via the local Langlands correspondence.

\par
To elaborate our result more rigorously, let $q$ be the cardinality of the residue field of $F$. The local symmetric square $L$-function $L(s,\pi,\mathrm{Sym}^2)$ is defined as the unique normalized generator of a $\mathbb{C}[q^{\pm s/2}]$-fractional ideal spanned by Rankin-Selberg integrals for the space of good sections. The generator is oftentimes referred as the greatest common divisor (gcd) \cite{Kaplan13}. Let $\mathrm{Sym}^2:GL_2(\mathbb{C}) \rightarrow GL_3(\mathbb{C})$ be the symmetric square representation. We can relate to the local symmetric square $L$-function $L(s,\mathrm{Sym}^2(\rho(\pi)))$ of Artin type, where $\rho$ stands for the local Langlands correspondence.

\begin{theorem*}[The equality]
\label{main1} Let $\pi$ be an irreducible admissible representation of $GL_2(F)$ and $\rho(\pi)$ the associated Langlands parameter. Then we have
\[
  L(s,\pi,\mathrm{Sym}^2)=L(s,\mathrm{Sym}^2(\rho(\pi))).
\]
\end{theorem*}

The discrete series case for $GL_n(F)$ has already been proven by the work of Yamana \cite{Ya17}. It was accomplished in  \cite{Henniart} that local symmetric square $L$-functions for $GL_n(F)$ from the Langlands-Shahidi method \cite{Shahidi} coincide with counterpart Artin $L$-functions. One immediate corollary of Theorem \ref{main1} is the factorization
\begin{equation}
\label{intro-factorization}
 L(s,\pi \times \pi)=L(s,\omega_{\pi})L(s,\pi,\mathrm{Sym}^2)=L(s,\pi,\wedge^2)L(s,\pi,\mathrm{Sym}^2)
  \end{equation}
where $\omega_{\pi}$ is the central character of $\pi$, and $L(s,\pi \times \pi)$, $L(s,\omega_{\pi})$ and $L(s,\pi,\wedge^2)$ denote, respectively, the Rankin-Selberg $L$-function for $GL_2(F) \times GL_2(F)$ \cite{JPSS}, the Tate $L$-function \cite{Tate}, and the Jacquet-Shalika  exterior square $L$-function for $GL_2(F)$ (See \cite{JO20-2}). As opposed to contending that \eqref{intro-factorization} is the definition in Gelbart and Jacquet \citelist{\cite{Gelbart-Jacquet}*{\S 3}}, we provide a natural way to define $L$-functions and express them directly in terms of inducing data.

\par
The second purpose of this note is to prove the stability of symmetric square local factors. Let $\psi$ be a fixed additive character of the field $F$. The Rankin-Selberg integrals satisfy a functional equation to define $\gamma$-factors $\gamma(s,\pi,\mathrm{Sym}^2,\psi)$. In contrast with the $\gamma$-factor being a rational function in $\mathbb{C}(q^{-s/2})$, an $\varepsilon$-factor $\varepsilon(s,\pi,\mathrm{Sym}^2,\psi)$ appearing in the functional equation is exponential, namely, a unit in $\mathbb{C}[q^{\pm s/2}]$.

\begin{theorem*}[The analytic stability] Let $\pi$ and $\sigma$ be irreducible admissible representations of $GL_2(F)$ sharing the same central character. For every sufficiently highly ramified characters $\chi$
of $F^{\times}$, identified as a character of $GL_2(F)$ through the determinant, we obtain
\[
  \gamma(s,\pi \otimes \chi,\mathrm{Sym}^2,\psi)=\gamma(s,\sigma \otimes \chi,\mathrm{Sym}^2,\psi).
\]
In this situation, the $L$ and epsilon factors stabilize as well:
\[
  L(s,\pi \otimes \chi, \mathrm{Sym}^2)=L(s,\sigma \otimes \chi, \mathrm{Sym}^2)=1 \quad \text{and} \quad
  \varepsilon(s,\pi \otimes \chi,\mathrm{Sym}^2,\psi)=\varepsilon(s,\sigma \otimes \chi,\mathrm{Sym}^2,\psi).
\]
\end{theorem*}

The analogous result has been settled for $GL_n$ over a non-archimedean field $F$ with characteristic zero \cite{CoShTs} and with positive characteristic \cite{Ganapathy-Lomeli} in the formulation of Langlands-Shahidi local coefficients.

\par

The Rankin-Selberg method plays a profound role in analyzing Langlands automorphic $L$-functions in company with the Langlands-Shahidi method \cite{CoShTs,Ganapathy-Lomeli,GaoShahidiSzpruch,Shahidi} and the doubling method \cite{HKS,Ya14}. The main subject of the Rankin-Selberg convolution is to determine the integral representation which admits a factorization into an Euler product. More importantly the global integral ought to produce the symmetric square $L$-function with unramified data. This method originated from the construction of Shimura \cite{Shimura} for $n=2$. Later this integral was reformulated in adelic language by Gelbart and Jacquet \cite{Gelbart-Jacquet} for $n=2$ and by Patterson and Piatetski-Shapiro \cite{PP} for $n=3$. In the early 1990's Bump and Ginzburg \cite{BuGi} extended the method to arbitrary $n$. Afterward Takeda \cite{TA14,Takeda15}
carried out the twisted case for general $n$. However the integrals of Gelbart-Jacquet and Takeda are not taken over $GL_2(F)$ so the family of integrals does not afford a $GL_2(F)$-trilinear form which is an essential ingredient to characterize the exceptional poles of $L$-function. For this reason, we mainly treat a modified local Zeta integral introduced by Yamana \cite{Ya17}.

\par
One significant difference from the setting of Jacquet, Piatetski-Shapiro, and Shalika \cite{JPSS}, is that in comparison to Fourier transforms, the functional equation involves the intertwining operator. In the literature one might be temped to try the space of holomorphic sections to generate a $\mathbb{C}[q^{\pm s/2}]$-fractional ideal but the poles of this fractional ideal only contribute the regular $L$-function. The notion of good sections appeared in the work of Piatetski-Shapiro and Rallis \cite{PSRA87}. Taking a lead of their methodology, the set of ``good sections" consists of holomorphic sections and the image of holomorphic sections under the normalized operator. In turn the additional pole coming from the normalized operator is attributed to the exceptional $L$-function. The noteworthy discovery is that adapting good sections compensates the lack of symmetry of the functional equation caused by just using holomorphic sections. For the sake of explaining it, the operator is a bijection on the space of good sections although holomorphic sections are not necessarily mapped to themselves.

\par
The terminologies of ``exceptional" and ``regular" parts of $L$-functions have been extensively exploited in the construction of $L$-functions for $GSp_4(F)$. We refer to \cite{Schmidt-Tran} and the references therein, which are all based on the work of Piatetski-Shapiro \cite{Piatetski-Shapiro}. The method of good sections was reshaped and rapidly developed in the perspective of the doubling method \cite{HKS,Ya14}. Nevertheless it takes several years for this approach to emerge in the study of non-archimedean local $L$-functions through integral representations. Kaplan \cite{Kaplan13} applied good sections to $L$-functions for $SO_{2m}(F) \times GL_n(F)$ and recently Chen \cite{Chen} implemented the study of Asai cube $L$-functions for $GL_2(F)$, which is precisely what is used by Piatetski-Shapiro and Rallis \cite{PSRA87} in the contexture of Rankin triple product for $GL_2(F)$. Thankfully the main result of \cite{JO20} asserts that $L$-functions supplied by auxiliary variables of Schwartz-Bruhat functions and good sections are in fact the same in the framework of numerous $GL_n(F)$-type cases such as local Rankin-Selberg, Asai, and exterior square $L$-functions.

\par
A major shortcoming of utilizing the Rankin-Selberg method is that one needs to prove multiplicativity of $\gamma$-factors. Cogdell and Piatetski-Shapiro \cite{Cogdell-PS} devised a systematic machinery to compute local $L$-functions without relying on multiplicative properties. The crux of their observation is to interpret the occurrence of poles of exceptional $L$-functions for Bernstein-Zelevinsky derivatives of representations \cite{BernsteinZelevinsky} as the appearance of various distinguished representations. In order to control the location of poles, they suggest suitably deforming representations for which $L$-functions are tractable. This technique is adapted to tackle the problem of computing  Asai \cite{Matringe09} and Bump-Friedberg $L$-functions \cite{Matringe} by Matringe and exterior square $L$-functions by the author \cite{JO-3,JO20-2}. See \cite{Chai} also for the derivatives and exceptional poles on archimedean places. Along the line of this prototype, Kable \cite{Kable01} initiated the project to understand the structure of the symmetric square $L$-function in the late 1990's. To gain an intimate knowledge of  that $L$-function, Kable was led to examine all the derivatives of exceptional representations. After that, the remaining task was the dyadic case and this computation was concluded by Kaplan \cite{Kaplan17} and Yamana \cite{Ya17} independently. As a continuation of their direction, we complete the particular case of $n=2$. It might be possible to reduce the local and global functorial lifts of Gelbart and Jacquet \cite{Gelbart-Jacquet} to our main results. However it is our belief that the context of the present paper will work out the higher ranked case of $GL_n(F)$ and we plan to do so in near future.


\par
The stability of symmetric square $L$-functions is known for $GL_2(F)$ \cite[(6.4)]{Gelbart-Jacquet}. Nonetheless the stability of $L$-functions does not directly imply the stability of the corresponding gamma and epsilon factors. In principle, our result should follow from the equality between Rankin-Selberg and Langlands-Shahidi local symmetric square $\gamma$-factors. To the author's understanding, the matching of two types of $\gamma$-factors is unfortunately not recorded anywhere. In this regard, a proof of stability of $\gamma$-factors within the context of integral representations has its own merit. Our approach in the proof is rooted on asymptotic analysis of partial Bessel functions associated with Howe vectors which can be viewed as an extension of the work \cite{ChaiZhang,Zhang} to the framework of the metaplectic group.

\par
The interested reader will notice that many of our calculation apply over any non-archimedean local field. The restriction to the characteristic on the field $F$ throughout this article comes from the intrinsic nature of the global integral built in number fields \cite{TA14,Takeda15}, which is sufficient for its global application.
Nevertheless the Bump and Ginzburg global integral \cite{BuGi} is actually constructed over (global) function fields. We pursue an in-depth investigation on the comparison of $L$-functions and stability of $\gamma$-factors in the positive characteristic case.

\par
The rest of the paper is organized as follows. Section \ref{sec2} contains the preliminary, including a brief review of the metaplectic group, exceptional representations and Rankin-Selberg integrals. Section \ref{sec3} is concerned with the exceptional and regular $L$-function. At the end of Section \ref{sec3} we deduce the factorization formula. The local symmetric square $L$-functions are computed in Section \ref{sec4} and Section \ref{sec5} is devoted to the stability of local factors.

\section{The Rankin-Selberg Integrals}
\label{sec2}
\subsection{The metaplectic group}
\label{sec2:1}
Let $F$ be a non-archimedean local field of characteristic zero with odd residual characteristic. We let $\mathcal{O}$ be the ring of integers of $F$, $\mathfrak{p}$ the unique prime ideal of $\mathcal{O}$, and $\varpi$ a uniformizer, so $\mathfrak{p}=(\varpi)$. We normalize the absolute value by $|\varpi|^{-1}=|\mathcal{O} \slash \mathfrak{p}|=q$. For each subgroup $H(F) \subset GL_2(F)$, we often write $H$ for $H(F)$ when the base field is clear from the context. Let $B=TN$ denote the Borel subgroup of upper triangular matrices, where
\[
  T=\left\{ t(a,b) := \begin{pmatrix} a & \\ & b \end{pmatrix} \; \middle| \; a, b \in F^{\times} \right\}
\]
is the maximal torus made of diagonal matrices and
\[
  N=\left\{ n(x) :=\begin{pmatrix} 1 & x \\ & 1 \end{pmatrix} \; \middle| \; x \in F \right\}
\]
is the unipotent radical of $B$. Let $Z$ denote the center of $GL_2$ and let $A$ denote the subtorus
\[
  A=\{ t(a,1) \; | \; a \in F^{\times} \}.
\]
Let
\[
 \overline{N}=\left\{ \overline{n}(x) :=\begin{pmatrix} 1 & \\ x & 1 \end{pmatrix} \; \middle| \; x \in F \right\}
\]
be the unipotent subgroup opposed to $N$ and $\overline{B}=T\overline{N}$ the lower triangular Borel subgroup. We write
\[
  w_2=\begin{pmatrix} & 1 \\ 1 & \end{pmatrix}
\]
to denote the long Weyl element in $GL_2$. Let $W$ denote the Weyl group defined by $W=N_{GL_2}(T)/T=\{I_2,w_2\}$. We recall the Bruhat decomposition
\[
  GL_2=B \cup Bw_2N,
\]
with uniqueness of expression, that is, every $g \not \in B$ has the unique expression of the form $g=bw_2n$, $b \in B, n \in N$. 
We denote by $P$ the mirabolic subgroup given by
\[
  P=\left\{  \begin{pmatrix} a & x \\ & 1 \end{pmatrix} \; \middle| \;  a \in F^{\times}, x \in F \right\} \cong A \ltimes N.
\]
Put $K=GL_2(\mathcal{O})$, the standard maximal compact subgroup of $GL_2(F)$. 

\par
The content of $\S$ \ref{sec2:1} and $\S$ \ref{sec2:2} is basically a summary of necessary definitions and essential results in \cite{TA14, Takeda15, Ya17}. 
For $a,a^{\prime}$ and $b \in F^{\times}$, the Hilbert symbol  is a map $(\cdot,\cdot)_F : F^{\times} \times F^{\times} \rightarrow \{ \pm 1 \}$ satisfying (See \citelist{\cite{Weil}*{Chapter VIII. \S 5}})
\begin{enumerate}[label=$(\mathrm{\arabic*})$]
\item $(a,b)_F(a^{\prime},b)_F=(aa^{\prime},b)_F$
\item $(a,b)_F(b,a)_F=1$
\item $(a,-a)_F(a,1-a)_F=1$
\item $\{a \; | \; (a,y)_F=1 \; \text{for all $y \in F^{\times}$}\}=(F^{\times})^2$, where $(F^{\times})^2=\{a^2 \; | \; a \in F^{\times} \}$.
\end{enumerate}
It is noteworthy that $(a,b)_F=1$ for all $a, b \in \mathcal{O}^{\times}$ if and only if $|2|=1$ (cf. \cite[Section 1.1.3]{Kaplan}). Hence the Hilbert symbol is unramified.
We set
\[
  \bm{X}\begin{pmatrix} a & b \\ c & d \end{pmatrix}=
  \begin{cases}
  c & \text{if $c\neq 0$,}\\
  d & \text{if $c=0$.}
  \end{cases}
\]
The Kubota $2$-cocycle $\sigma_2 : GL_2 \times GL_2 \rightarrow \{\pm 1 \}$ is defined by
\[
  \sigma_2(g_1,g_2)=\left( \mathrm{det}(g_1),\frac{\bm{X}(g_1g_2)}{\bm{X}(g_1)} \right)_F \left( \frac{\bm{X}(g_1g_2)}{\bm{X}(g_1)}, \frac{\bm{X}(g_1g_2)}{\bm{X}(g_2)} \right)_F.
\]

\par
The \textit{metaplectic double cover} $\widetilde{GL}_2$ is a non-trivial central extension of $GL_2$ by $\{ \pm 1 \}$:
\[
  1 \longrightarrow \{ \pm 1 \} \longrightarrow  \widetilde{GL}_2 \overset{pr}{\longrightarrow} GL_2 \longrightarrow 1,
\]
where $pr$ is a canonical projection given by $pr(g,\xi)=g$ for $g \in GL_2$ and $\xi \in \{ \pm 1\}$. As a set, $\widetilde{GL}_2$ is realized to be 
\[
 \widetilde{GL}_2=GL_2(F) \times \{ \pm 1 \}=\{  (g,\xi) \; | \; g \in GL_2(F), \xi \in \{ \pm 1 \} \}
\]
and the group law is defined by
\[
  (g_1,\xi_1) \cdot (g_2,\xi_2)=(g_1g_2,\sigma_2(g_1,g_2)\xi_1\xi_2).
\]
It is known that there exists a compact subgroup $\mathcal{K}$ of $GL_2$ which splits in $\widetilde{GL}_2$  (cf. \citelist{\cite{TA14}*{\S 1.1} \cite{Ya17}*{1B}}), that is to say, there is a continuous map $s_2 : GL_2 \rightarrow \{ \pm 1 \}$ such that $\sigma_2(k_1,k_2)=s_2(k_1)s_2(k_2)s_2(k_1k_2)$ for all $k_1,k_2 \in \mathcal{K}$. If the residue characteristic of $F$ is odd (that being said, $|2|=1$) then we can take $\mathcal{K}=K$ (See \citelist{\cite{TA14}*{\S 1.1}}) and $s_2(k_1)s_2(k_2)s_2(k_1k_2)=1$ for all $k_1,k_2 \in K$. With this choice of $s_2$, the section $\mathcal{K} \rightarrow \widetilde{GL}_2$ defined by $k \mapsto (k,s_2(k))$ is what is called the canonical lift of Kazhdan and Patterson \cite{KaPa}. We define another $2$-cocycle $\tau_2$ by
\[
 \tau_2(g_1,g_2)=\sigma_2(g_1,g_2)s_2(g_1)s_2(g_2)s_2(g_1g_2) \;\; \text{for $g_1, g_2 \in GL_2$.}
\]
The choice of $s_2$, and hence $\tau_2$, is not unique. However as explained in \cite[P.181]{TA14}, we assume that $s_2$ is chosen to be trivial so that $\tau_2$ coincides with $\sigma_2$. We define a set theoretic section $\mathfrak{s}$ by $\mathfrak{s}(g)=(g,1)$ (cf. \cite[(1.1.4)]{Gelbart-PS}). We remark that the multiplication in the image $\mathfrak{s}(GL_2)$ is given via $\sigma_2$, by means of
\[
  (g_1,1)(g_2,1)=(g_1g_2,\sigma_2(g_1,g_2)) \;\; \text{for $g_1, g_2 \in GL_2$.}
\]
For every $m$, let $K_m$ be the $m$-th congruence subgroup, that is, $K_m=\{k\in K \; | \; k \equiv I_2 \;(\mathrm{mod}\; \mathfrak{p}^m) \}$. Then the collection $\{ \mathfrak{s}(K_m) \; | \; m \geq 0 \}$ is a basis of compact open neighborhoods of the identity in $\widetilde{GL}_2$ and the topology of $\widetilde{GL}_2$ as a locally compact group is determined by the embedding $\mathfrak{s} : K \rightarrow \widetilde{GL}_2$.


\par
We introduce the basic property of Hilbert symbol. For $g \in GL_2$, $g=n_1twn_2$ for some $n_1, n_2\in N$, $t \in T$, and $w \in W$.  
We define $\mathfrak{t}$ to be the map $GL_2 \rightarrow T$, given by $\mathfrak{t}(g)=\mathfrak{t}(n_1twn_2)=t$.

\begin{lemma}\citelist{\cite{BLS}*{Theorem 7}}
\label{BLS}
Suppose that $n \in N$, $g \in GL_2$ and $w \in W$. Let $t=t(a,b)$ and $t^{\prime}=t(a^{\prime},b^{\prime})$ be elements in $T$. Then we have
\begin{enumerate}[label=$(\mathrm{\arabic*})$]
\item\label{BLS-1} $\sigma_2(n,g)=\sigma_2(g,n)=1$ for all $n \in N,\; g \in GL_2$
\item\label{BLS-2} $\sigma_2(t,t^{\prime})=(a,b^{\prime})_F$ for all $t,\; t^{\prime} \in T$
\item\label{BLS-3} $\sigma_2(t,g)=\sigma_2(t,\mathfrak{t}(g))$ for all $t \in T,\; g \in GL_2$
\item\label{BLS-4} $\sigma_2(w,g)=\sigma_2(\mathfrak{t}(wg)\mathfrak{t}(g)^{-1},-\mathfrak{t}(g))$ for all $w \in W, \; g \in GL_2$.
\end{enumerate}
\end{lemma}
 
 \ref{BLS-3} and  \ref{BLS-4} reduces the calculation to  \ref{BLS-2} and then we compute the cocycle explicitly by Hilbert symbols. Then we deduce the following general formula.
 
\begin{lemma}\citelist{\cite{BLS}*{Theorem 7}}
Let $g, g^{\prime} \in GL_2$ and suppose $g=n_1twn_2$. Then
\[
\sigma_2(g,g^{\prime})=\sigma_2(t,wn_2g^{\prime})\sigma_2(w,n_2g^{\prime}).
\]
\end{lemma}
The right-hand side can be computed by \ref{BLS-3} and  \ref{BLS-4} of Lemma \ref{BLS}. For each subgroup $H \subset GL_2$, we let $\widetilde{H}=pr^{-1}(H)$ the metaplectic preimage of $H$ under $pr$. $\mathfrak{s}$ splits $H$ whenever the cocycle $\sigma_2$ is trivial on $H \times H$. In this case we simply denote by $H^{\ast}$ the image $\mathfrak{s}(H)$. Then $\widetilde{H}$ is the direct product of $\{ \pm 1 \}$ with $H^{\ast}$. We see from Lemma \ref{BLS} that $\mathfrak{s}$ splits the following subgroups $N$, $A$, and
\[
  Z^2=\left\{ \begin{pmatrix} a & \\ & a \end{pmatrix} \; \middle| \; a \in (F^{\times})^2 \right\}.
\]
Accordingly we denote $\mathfrak{s}(N)$, $\mathfrak{s}(A)$, and $\mathfrak{s}(Z^2)$ by $N^{\ast}$, $A^{\ast}$, and ${Z^2}^{\ast}$. In particular $\mathfrak{s}$ splits $W$ if and only if $(-1,-1)_F=1$ (See \cite[Section 5]{BLS}). We note that $\widetilde{GL}_1=GL_1 \times \{ \pm 1 \}$, where the product is given by the direct product. Also we define $\widetilde{F^{\times}}$ to be $\widetilde{F^{\times}}=F^{\times} \times \{ \pm 1 \}$
as a set but the product is $(a_1,\xi_1)\cdot(a_2,\xi_2)=(a_1a_2,(a_1,a_2)_F\xi_1\xi_2)$. 

\par
We know from Lemma \ref{BLS} that $\sigma_2(a_1I_2,a_2I_2)=(a_1,a_2)_F$. Hence $\widetilde{Z}$ is not the center of $\widetilde{GL}_2$. As a matter of the fact, $\widetilde{Z}^2$ is the center of $\widetilde{GL}_2$ and $\widetilde{Z}$ is isomorphic to $\widetilde{F^{\times}}$. We also note that $\widetilde{Z}$ is the center of $\widetilde{GL}_2^{(2)}$, where $\widetilde{GL}_2^{(2)}$ is the metaplectic preimage of
\[
 GL_2^{(2)}=\{ g \in GL_2 \; | \; \mathrm{det}(g) \in (F^{\times})^2 \}.
\]

\par
We fix a non-trivial continuous additive character $\psi$ of $F$. We define $\mathfrak{f}(\psi)$, the conductor of $\psi$, to be the smallest positive integer $m$ such that $\psi$ is trivial on $\mathfrak{p}^m$. For each $a \in F^{\times}$, we denote by $\psi_a$ the additive character defined by $\psi_a(x)=\psi(ax)$. The map $F \rightarrow \mathbb{C}^{\times}$ defined by $x \mapsto \psi(x^2)$ is what Weil called a {\it character of second degree}. The Weil index $\gamma(\psi)$ of $\psi$ \cite{Weil} is an eighth root of unity attached to any character of second degree $\psi$. Likewise we can define $\gamma(\psi_a)$ for each $a \in F$. We put
\[
 \mu_{\psi}(a)=\frac{\gamma(\psi_a)}{\gamma(\psi)}.
\]
What is particularly important is that
\[
  \mu_{\psi}(ab)=\mu_{\psi}(a)\mu_{\psi}(b)(a,b)_F \quad \text{and} \quad \mu_{\psi}(ab^2)=\mu_{\psi}(a)
\]
for $a, b \in F^{\times}$. Hence it extends to a group homomorphism $\widetilde{F^{\times}} \rightarrow \mathbb{C}$ defined by $(a,\xi) \mapsto \xi\mu_{\psi}(a)$.

\subsection{Exceptional representations}
\label{sec2:2}
In \S \ref{sec2:2}, we review some technical results on the exceptional representations. Let $H$ be a subgroup of $GL_2$. We let $\textbf{1}_{H}$ denote the trivial character on $H$. Let $\pi$ denote an irreducible admissible representation of $\widetilde{H}$. $\pi$ is said to be \textit{genuine} if $\pi((1,\xi)h)=\xi \pi(h)$ for all $\xi \in \{ \pm 1 \}$ and $h \in H$, that is, each element in $(1,\xi) \in \widetilde{H}$ acts as a multiplication by $\xi$. Any representation $\pi$ of $H$ can be pulled back to a non-genuine representation of $\widetilde{H}$ by composing it with the canonical projection $pr : \widetilde{GL}_2 \rightarrow GL_2$. In particular, for a Borel subgroup $B$, we view the modular character of $\delta_B$ as a character on $\widetilde{B}$ in this way.

\par
As we discussed before, each element in $\widetilde{N}$ can be written in the form $(1,\xi)n^{\ast}$ for $n^{\ast} \in N^{\ast}$ and $\xi \in \{ \pm 1 \}$ and $(1,\xi) \in \widetilde{T}$. We can check by exploiting Lemma \ref{BLS} that $\mathfrak{s}(t)\mathfrak{s}(n)\mathfrak{s}(t)^{-1}=\mathfrak{s}(tnt^{-1})$  for all $t \in T$ and $n \in N$. Thus $\widetilde{B}=\widetilde{T}N^{\ast}$ and
$N^{\ast}$ is normalized by $\widetilde{T}$. We also have $\widetilde{T} \cap N^{\ast}=\{ (1,1) \}$. For the maximal torus $T \subset B$, we let
\[
  T^e=\left\{ \begin{pmatrix} a & \\ & b \end{pmatrix} \; \middle| \; \text{$ab^{-1}$ is square } \right\}.
\]
The metaplectic preimage $\widetilde{T}^e$ is a maximal abelian subgroup of $\widetilde{T}$. We define a character $\omega^{\psi}$ on $\widetilde{T}^e$ by
$\omega^{\psi}((1,\xi)\mathfrak{s}(t(a,b)))=\xi\mu_{\psi}(b)^{-1}$. For $t(a,b)$ and $t(a^{\prime},b^{\prime}) \in \widetilde{T}^e$, Lemme \ref{BLS} implies that $\sigma_2(t(a,b),t(a^{\prime},b^{\prime}))=(a,a^{\prime})_F$. We can conclude that $\omega^{\psi}$ is indeed a genuine character of $\widetilde{T}^e$. The {\it exceptional representation} $\theta^{\psi}$ \cite{KaPa} is the unique irreducible quotient of the normalized induced representation $\mathrm{Ind}^{\widetilde{GL}_2}_{\widetilde{T}^e N^{\ast}}(\omega^{\psi} \otimes \delta_{B}^{1/4})$, isomorphic to the unique irreducible subrepresentation of $\mathrm{Ind}^{\widetilde{GL}_2}_{\widetilde{T}^eN^{\ast}}(\omega^{\psi} \otimes \delta_{B}^{-1/4})$. Here normalized induction means that $\mathrm{Ind}^{\widetilde{GL}_2}_{\widetilde{T}^e N^{\ast}}(\omega^{\psi} \otimes \delta_{B}^{1/4})$ is unitarizable whenever $\omega^{\psi} \otimes \delta_{B}^{1/4}$ is unitarizable. 
Let $\widetilde{T}^2$ be the inverse image of $T^2$, where $T^2=\{t^2 \; | \; t \in T \}$ is the set of square elements in the torus. $\widetilde{T}^2\widetilde{Z}^2$ is the center of $\widetilde{T}$. In general a character $\omega$ of $\widetilde{T}^2\widetilde{Z}^2$ is called {\it exceptional} if $\omega(\mathfrak{s}(x^2,x^{-2}))=|x|$ for all $x \in F^{\times}$. The restriction of $\omega^{\psi} \otimes \delta_B^{1/4}$ to $\widetilde{T}^2\widetilde{Z}^2$ is an exceptional character.

\par
Let $\eta$ be a character of $F^{\times}$. For a ramified character $\eta$, let $\mathfrak{f}(\eta)$ be the conductor of $\eta$ defined to be the smallest integer $m$ such that $\eta$ is trivial on $1+\mathfrak{p}^m$. A character $\eta$ of $F^{\times}$ is said to be {\it quadratic} if $\eta^2=1$. 
We define a genuine character $\widetilde{\eta}$ of $\widetilde{Z}^2$ by
\[
  \widetilde{\eta}((1,\xi)\mathfrak{s}(z))=\xi\eta(a), \quad z=\begin{pmatrix} a & \\ & a \end{pmatrix} \in Z^2.
\]
We embed $GL_1$ into $GL_2$ via the map $a \mapsto \begin{pmatrix} a & \\ & 1 \end{pmatrix}$. Then we extend $\textbf{1}_{GL_1}$ to the representation $\widetilde{\textbf{1}}_{GL_1}\boxtimes \widetilde{\eta}$ of the semidirect product $(\widetilde{GL_1}\times \widetilde{Z}^2) \ltimes N^{\ast}$ by letting $\widetilde{Z}^2$ act by $\widetilde{\eta}$ and $N^{\ast}$ act trivially, where $\widetilde{\textbf{1}}_{GL_1}$ is the non-genuine character on $\widetilde{GL}_1$ (trivial extension) given by $(a,\xi) \mapsto \xi$. For $s \in \mathbb{C}$, we define a normalized induced representation (cf. \cite[p.132]{Takeda15})
\[
  I(s,\eta)=\mathrm{Ind}^{\widetilde{GL}_2}_{\widetilde{Z}^2\widetilde{P}}((\widetilde{\eta}   \boxtimes \widetilde{\textbf{1}}_{GL_1}  ) \otimes \delta^{s/4}_{B}).
\]
equipped with the natural action of $\widetilde{GL}_2$ on $ I(s,\eta)$ by a right translation $R$. Let us look at the transformation under $\widetilde{Z}^2$. We write $(1,\xi)\mathfrak{s}(z) \in \widetilde{Z}^2$
with $z=aI_2$ and $a \in (F^{\times})^2$. For $f_s \in I(s,\eta)$ we have
\begin{equation}
\label{central-induced}
f_s((1,\xi)\mathfrak{s}(z))=\xi \eta(a)  f_s(I_2).
\end{equation}
The two sides of the functional equation in \eqref{unnoraml-intertwining} involve slightly different induced representations. For $w \in W$, we denote $\mathfrak{s}(w)$ simply by the same symbol $w$ when there is no danger of confusion. We construct the induced representation occurring the left hand side of the functional equation. In our situation, it means that
\[
  J(-s,\eta)=\mathrm{Ind}^{\widetilde{GL}_2}_{\widetilde{Z}^2\;{^{w_2}\widetilde{A}}N^{\ast}}((\widetilde{\eta} \boxtimes {^{w_2}\widetilde{\textbf{1}}}_{GL_1}  ) \otimes \delta^{-s/4}_{B}),
\]
where the twisted representation ${^{w_2}\widetilde{\textbf{1}}}_{GL_2}$ of $\widetilde{\textbf{1}}_{GL_2}$, to be the representation of ${^{w_2}\widetilde{A}}=w_2\widetilde{A}w_2^{-1}$, is given by ${^{w_2}\widetilde{\textbf{1}}_{GL_1}(\widetilde{a})}=\widetilde{\textbf{1}}_{GL_1}(w_2^{-1}\widetilde{a}w_2)$ for $\widetilde{a} \in {^{w_2}\widetilde{A}}$.

\subsection{Rankin-Selberg integrals and good sections} 
\label{sec2:3}
We investigate the basic definition of $L$-functions and basic existence theorems. Everything stated in $\S$ \ref{sec2:3} without any specific reference is found in \cite[\S 3]{Ya17}. Let $\pi$ be an irreducible admissible generic representation of $GL_2$, $\omega_{\pi}$ its central character, and $\mathcal{W}(\pi,\psi)$ its Whittaker model. Let $\theta^{\psi}$ denote an exceptional representation of $\widetilde{GL}_2$. Then there exists a unique non-zero Whittaker functional $\lambda$ on $V_{\theta^{\psi}}$ such that 
\begin{equation}
\label{central-theta}
\lambda \left[\theta^{\psi}\left((1,\xi)\mathfrak{s}\begin{pmatrix} a & \\ & a \end{pmatrix}\mathfrak{s}\begin{pmatrix}1 & x \\ & 1 \end{pmatrix}\right)v\right]
=\xi\mu_{\psi}(a)^{-1}\psi^{-1}(x)\lambda(v)
\end{equation}
for $a \in F^{\times}$, $x \in F$ and $v \in V_{\theta^{\psi}}$ \cite[\S 2.1]{Gelbart-PS}. We let $\mathcal{W}(\theta^{\psi},\psi^{-1})$ denote the Whittaker model of $\theta^{\psi}$. We obtain the Whittaker model $\mathcal{W}(\theta^{\psi},\psi^{-1})$ by setting $(W_{\theta^{\psi}})_v(\widetilde{g})=\lambda(\theta^{\psi}(\widetilde{g})v)$ for $v \in V_{\theta^{\psi}}$ and $\widetilde{g} \in \widetilde{GL}_2$. We put $W_{\theta^{\psi}}=(W_{\theta^{\psi}})_v$.

\par
For $W \in \mathcal{W}(\pi,\psi)$, $W_{\theta^{\psi}} \in \mathcal{W}(\theta^{\psi},\psi^{-1})$, and $f_{2s-1} \in I(2s-1,\omega^{-1}_{\pi})$, we associate the Zeta integral \cite{TA14, Ya17}
\begin{equation}
\label{symsquare-RS}
  I(W,W_{\theta^{\psi}},f_{2s-1})=\int_{Z^2N \backslash GL_2} W(g)W_{\theta^{\psi}}(\mathfrak{s}(g)) f_{2s-1}(\mathfrak{s}(g)) dg.
\end{equation}
This integral is absolute convergent for $\mathrm{Re}(s)$ sufficiently large. By means of Properties \eqref{central-induced} and \eqref{central-theta}, the group $Z^2$ acts on the product $W_{\theta^{\psi}}(-)f_{2s-1}(-)$ as $\omega_{\pi}^{-1}$. Therefore the integral $I(W,W_{\theta^{\psi}},f_{2s-1})$ is well-defined in the sense that $Z^2$ acts trivially for the integrand.

\par
There is an intertwining operator
\[
  M(s,\eta) : I(s,\eta) \rightarrow J(-s,\eta)
\]
given by the formula
\[
  M(s,\eta)f_s(\widetilde{g})=\int_{F} f_s\left( \mathfrak{s} \begin{pmatrix} & 1 \\ 1& \end{pmatrix} \mathfrak{s} \begin{pmatrix} 1 & x \\ & 1 \end{pmatrix} \widetilde{g} \right) \; dx.
\]
This integral converges absolutely for $\mathrm{Re}(s)$ large and is defined by meromorphic continuation otherwise. The intertwining operator is mainly utilized later in \S \ref{sec5:2}.
 Now it is proven in \cite[Proposition 3.14]{Ya17} that there exists a rational function in $\mathbb{C}(q^{-s/2})$ enjoying the following functional equation
\begin{equation}
\label{unnoraml-intertwining}
  I(W,W_{\theta^{\psi}},M(2s-1,\omega_{\pi}^{-1})f_{2s-1})=\Gamma(s,\pi,\mathrm{Sym}^2,\psi) I(W,W_{\theta^{\psi}},f_{2s-1}).
\end{equation}
As we have seen in \cite{Ya17}, we normalize the intertwining operator by
\[
 \hat{N}(s,\eta,\psi)= \gamma(s,\eta^{-2},\psi)M(s,\eta).
\]
We want to understand an involution of group $g \mapsto {^{\iota}g}$ of $GL_2$ defined by ${^{\iota}g}:=w_2\; {^tg}^{-1} w_2$. An automorphism $^{\iota} : \widetilde{g} \mapsto {^{\iota}\widetilde{g}}$ is called a lift of the involution if ${^{\iota}\xi}=\xi$ and $p({^{\iota}\widetilde{g}})={^{\iota}p(\widetilde{g})}$ for all $\xi \in \{ \pm 1 \}$ and $\widetilde{g} \in \widetilde{GL}_2$. Kable \citelist{\cite{Kable99}\cite{Ya17}*{Proposition 1.3}} constructed a lift $\widetilde{g} \mapsto {^{\iota}\widetilde{g}}$ of $g \mapsto {^{\iota}g}$ to $\widetilde{GL}_2$ satisfying
\[
  ^{\iota}\mathfrak{s}(t(a,b))=\mathfrak{s}(^{\iota}t(a,b))(b,a)_F, \quad {^{\iota}\widetilde{z}}=\widetilde{z}^{-1}, \quad {^{\iota}(^{\iota}\widetilde{g})}=\widetilde{g}, \quad
  ^{\iota}\mathfrak{s}(n)=\mathfrak{s}(^{\iota}n)
\]
for all $t(a,b) \in T$, $\widetilde{z} \in \widetilde{Z}^2$, $\widetilde{g} \in \widetilde{GL}_2$ and $n \in N$. Furthermore if the residual characteristic of $F$ is odd and $f :K \rightarrow \widetilde{GL}_2$ is a homomorphism, then $f(^{\iota}k)={^{\iota}f}(k)$ for all $k \in K$. Following Yamana \cite{Ya17}, we then define a $\mathbb{C}$-linear map 
\[
   N(s,\eta,\psi) : I(s,\eta) \rightarrow I(-s,\eta^{-1})
\]
by setting $N(s,\eta,\psi)f_s(\widetilde{g})={^{\iota}[{\hat{N}}(s,\eta,\psi)f_s]}(\widetilde{g})=[\hat{N}(s,\eta,\psi)f_s]({^{\iota}\widetilde{g}})$. The normalized operator satisfies the functional equation.

\begin{proposition}
\label{intertwining-func}
Let $N(s,\eta,\psi)$ be a normalized operator as above. Then we have
\[
N(-s,\eta^{-1},\psi^{-1}) \circ N(s,\eta,\psi)=\mathrm{Id}.
\] 
\end{proposition}

\begin{proof}
The induced representation $I(s,\eta)$ can be viewed as a subrepresentation of a genuine principal series representation $\mathrm{Ind}^{\widetilde{GL}_2}_{\widetilde{T}^e N^{\ast}}(\mu_s)$, where $\mu_s$ is an extension to $\widetilde{T}^e$ of the genuine character of $\widetilde{T}^2\widetilde{Z}^2$ defined by
\[
  \mu_s\left( \mathfrak{s} \begin{pmatrix}a & \\ & b \end{pmatrix} \right)=\eta(b)\mu_{\psi}(b)|a|^{\frac{s}{4}}|b|^{-\frac{s}{4}}=\eta(b)\mu_{\psi}(b)\delta_B^{\frac{s}{4}} \begin{pmatrix}a & \\ & b \end{pmatrix}
\]
for $t(a,b) \in T^2Z^2$. 
Shahidi \cite{Ganapathy-Lomeli,GaoShahidiSzpruch,Shahidi} defines the Plancherel measure $\mu(s,\eta)$ associated with $\eta$ by 
\begin{equation}
\label{plancherel}
M(-s,\eta) \circ M(s,\eta)=\mu(s,\eta)^{-1} \cdot \mathrm{Id}. 
\end{equation}
It is a priori a rational function in $\mathbb{C}(q^{-s/4})$. As described in \cite[(4.9),(4.11),(9.22)]{GaoShahidiSzpruch}, the formula we seek for $\mu(s,\eta)^{-1}$ is therefore
\[
\begin{split}
q^{\mathfrak{f}(\psi)-\mathfrak{f}(\eta^2)}\frac{L(s,\eta^{-2})L(-s,\eta^{2})}{L(1-s,\eta^2)L(1+s,\eta^{-2})}
&=\varepsilon(s,\eta^{-2},\psi)^{-1}\varepsilon(-s,\eta^2,\psi^{-1})^{-1}\frac{L(s,\eta^{-2})L(-s,\eta^{2})}{L(1-s,\eta^2)L(1+s,\eta^{-2})}\\
&= \gamma(s,\eta^{-2},\psi)^{-1}\gamma(-s,\eta^{2},\psi^{-1})^{-1}
\end{split}
\] 
and the rational function $\alpha_{\psi}(s,\eta)$ in \cite[Lemma 3.5]{Ya17} is given by 
\[
\varepsilon(s,\eta^{-2},\psi)^{-1}\varepsilon(-s,\eta^2,\psi^{-1})^{-1}=q^{\mathfrak{f}(\psi)-\mathfrak{f}(\eta^2)}.
\]
Indeed the formula is originally stated for $\widetilde{SL}_2$ \cite[(9.22)]{GaoShahidiSzpruch}. Nevertheless the very recent result \cite[Corollary 10.2]{GaoShahidiSzpruch} allows us to relate the Plancherel measure associated to a genuine representation of $\widetilde{GL}_2$ with that of $\widetilde{SL}_2$. If we incorporate this into normalized operators, \eqref{plancherel} can be rewritten as $N(-s,\eta^{-1},\psi^{-1}) \circ N(s,\eta,\psi)=\mathrm{Id}$.
\end{proof}

\par
In what follows we refer to Waldspurger \cite[\S 4]{Wal03} for a through treatment of the notion of sections. We observe that $\widetilde{K} \simeq K^{\ast} \times \{ \pm 1 \}$ is a compact open subgroup \cite[\S 5]{Kable99}. A $\widetilde{K}$-finite function $f : \mathbb{C} \times \widetilde{GL}_2 \rightarrow \mathbb{C}$ such that the mapping $\widetilde{g} \mapsto f(s,\widetilde{g})$ belonging to $I(s,\eta)$ for all $s$ is called a \textit{section}. A section $f_s \in I(s,\eta)$ is called a \textit{standard section} if its restriction to $\widetilde{K}$ is independent of $s$. Let $V_{std}(s,\eta)$ denote the space of standard sections. The space of \textit{holomorphic sections} is defined by $V_{hol}(s,\eta)=\mathbb{C}[q^{-s/4},q^{s/4}]\otimes_{\mathbb{C}}V_{std}(s,\eta)$. The elements of $V_{rat}(s,\eta)=\mathbb{C}(q^{-s/4})\otimes_{\mathbb{C}}V_{std}(s,\eta)$ are called \textit{rational sections}. In order to incorporate sections in normalized operator, we need to allow sections to vary arithmetically in $s$. To this end, Piatetski-Shapiro and Rallis introduce the family of good sections \cite{PSRA87}.

\begin{definition}
\label{GoodSection}
We define the space $V_{good}(2s-1,\eta)$ of good sections to comprise the following:
\begin{enumerate}[label=$(\mathrm{\roman*})$]
\item $V_{hol}(2s-1,\eta)$
\item $N(1-2s,\eta^{-1},\psi^{-1})\left[ V_{hol}(1-2s,\eta^{-1}) \right]$.
\end{enumerate}
\end{definition}

According to Proposition 3.11 $(4)$ of \cite{Ya17}, Definition \ref{GoodSection} agrees with that of Yamana when $\eta$ is unitary. The good section is closed under the normalized intertwining operator in the sense that $N(2s-1,\eta,\psi)\left[V_{good}(2s-1,\eta)\right] \subset V_{good}(1-2s,\eta^{-1})$. If $f_{2s-1} \in V_{good}(2s-1,\omega^{-1}_{\pi})$, then $I(W,W_{\theta^{\psi}},f_{2s-1})$ is a rational function of $q^{-s/2}$. Let $\mathcal{I}(\pi)$ be a subspace of $\mathbb{C}(q^{-s/2})$ generated by integrals $I(W,W_{\theta^{\psi}},f_{2s-1})$ for $W \in \mathcal{W}(\pi,\psi)$, $W_{\theta^{\psi}} \in \mathcal{W}(\theta^{\psi},\psi^{-1})$, and $f_{2s-1} \in V_{good}(2s-1,\omega^{-1}_{\pi})$. Each such integral can be written with a common denominator (cf. \cite[Proposition 3.8-(3)]{Ya17}) and $\mathcal{I}(\pi)$ is closed under multiplication by $q^{s/2}$. Then $\mathcal{I}(\pi)$ is a $\mathbb{C}[q^{-s/2},q^{s/2}]$-fractional ideal and in particular it contains $1$. Hence we can always find a normalized generator of the form $P(q^{-s/2})^{-1}$ where the polynomial $P(X) \in \mathbb{C}[X]$ satisfies $P(0)=1$.

\begin{definition}
\label{symsquare-dfn}
We define the \textit{symmetric square $L$-function} by $L(s,\pi,\mathrm{Sym}^2)=P(q^{-s/2})^{-1}$, the normalized generator of the fractional ideal $\mathcal{I}(\pi)$
formed by the family of integrals $I(W,W_{\theta^{\psi}},f_{2s-1})$ for $W \in \mathcal{W}(\pi,\psi)$, $W_{\theta^{\psi}} \in \mathcal{W}(\theta^{\psi},\psi^{-1})$, and $f_{2s-1} \in V_{good}(2s-1,\omega^{-1}_{\pi})$.
\end{definition}

Let $\pi^{\iota}$ denote the representation of $GL_2$ on the same space $V_{\pi}$ but with the action ${\pi^{\iota}}(g)=\pi({^{\iota}g})$. If $\pi$ is irreducible, then ${\pi^{\iota}}=\widetilde{\pi}$, the \textit{contragredient representation} (See \cite[\S 2.1]{Cogdell-PS}). Define $\widetilde{W} : GL_2 \rightarrow \mathbb{C}$ as $\widetilde{W}(g)=W({^{\iota}g})=W(w_2{^tg^{-1}}w_2)$. Then we have $\widetilde{W} \in \mathcal{W}(\widetilde{\pi},\psi^{-1})$. For $W_{\theta^{\psi}} \in \mathcal{W}(\theta^{\psi},\psi^{-1})$, we define a $\psi$-Whittaker function $\widetilde{W}_{\theta^{\psi}}$ by $\widetilde{W}_{\theta^{\psi}}(\widetilde{g})=W_{\theta^{\psi}}({^{\iota}\widetilde{g}})$ for $\widetilde{g} \in \widetilde{GL}_2$. We note that $\widetilde{W}_{\theta^{\psi}} \in \mathcal{W}(\theta^{\psi^{-1}},\psi)$ (See \cite[Lemma 1.9 (4)]{Ya17}). With the normalized operator, there is a rational function $\gamma(s,\pi,\mathrm{Sym}^2,\psi)$ in $\mathbb{C}(q^{-s/2})$ such that we have
\[
 I(\widetilde{W},\widetilde{W}_{\theta^{\psi}},N(2s-1,\omega^{-1}_{\pi},\psi)f_{2s-1})=\gamma(s,\pi,\mathrm{Sym}^2,\psi)I(W,W_{\theta^{\psi}},f_{2s-1}).
\]
The $\varepsilon$-factor is defined as the ratio
\begin{equation}
\label{epsilon}
  \varepsilon(s,\pi,\mathrm{Sym}^2,\psi)=\frac{\gamma(s,\pi,\mathrm{Sym}^2,\psi)L(s,\pi,\mathrm{Sym}^2)}{L(1-s,\widetilde{\pi},\mathrm{Sym}^2)}.
\end{equation}
With the $\varepsilon$-factor in hand, the functional equation can be written in the form
\begin{equation}
\label{functional-equation}
  \frac{I(\widetilde{W},\widetilde{W}_{\theta^{\psi}},N(2s-1,\omega^{-1}_{\pi},\psi)f_{2s-1})}{L(1-s,\widetilde{\pi},\mathrm{Sym}^2)}
  =\varepsilon(s,\pi,\mathrm{Sym}^2,\psi) \frac{I(W,W_{\theta^{\psi}},f_{2s-1})}{L(s,\pi,\mathrm{Sym}^2)}.
\end{equation}
Converting $I(\widetilde{W},\widetilde{W}_{\theta^{\psi}},N(2s-1,\omega^{-1}_{\pi},\psi)f_{2s-1})$ on the left hand side of \eqref{functional-equation} as in the proof of Proposition 3.14 of \cite{Ya17} , the function equation \eqref{functional-equation} becomes
\begin{equation}
\label{func-Nhat}
  \frac{I(W,W_{\theta^{\psi}},\hat{N}(2s-1,\omega^{-1}_{\pi},\psi)f_{2s-1})}{L(1-s,\widetilde{\pi},\mathrm{Sym}^2)}
  =\varepsilon(s,\pi,\mathrm{Sym}^2,\psi) \frac{I(W,W_{\theta^{\psi}},f_{2s-1})}{L(s,\pi,\mathrm{Sym}^2)}.
\end{equation}

\begin{proposition}
\label{func-epsilon}
The epsilon factor $\varepsilon(s,\pi,\mathrm{Sym}^2,\psi)$ satisfies the following functional equation;
\[
\varepsilon(1-s,\widetilde{\pi},\mathrm{Sym}^2,\psi^{-1}) \varepsilon(s,\pi,\mathrm{Sym}^2,\psi)=1. 
\]
Furthermore $\varepsilon(s,\pi,\mathrm{Sym}^2,\psi)$ is a unit in $\mathbb{C}[q^{\pm s/2}]$.
\end{proposition}

\begin{proof}
The proof is akin to that of \cite[Theorem 3.11]{JO20}. We omit the complete details.
\end{proof}

Comparing \eqref{unnoraml-intertwining} with \eqref{func-Nhat}, the two factors $\gamma(s,\pi,\mathrm{Sym}^2,\psi)$ and $\Gamma(s,\pi,\mathrm{Sym}^2,\psi)$ are related by the following way.

\begin{proposition}
\label{division}
As functions in $\mathbb{C}(q^{-s/2})$, we have
\[
 \Gamma(s,\pi,\mathrm{Sym}^2,\psi)  =\frac{\gamma(s,\pi,\mathrm{Sym}^2,\psi) }{\gamma(2s-1,\omega^2_{\pi},\psi)}.
\]
\end{proposition}

We return to the integral $I(W,W_{\theta^{\psi}},f_{2s-1})$ in \eqref{symsquare-RS}. Since $W_{\theta^{\psi}}$ and  $f_{2s-1}$ are genuine, their product $\widetilde{g} \mapsto W_{\theta^{\psi}}(\widetilde{g}) f_{2s-1}(\widetilde{g})$ must factor through the natural map $\widetilde{GL}_2 \rightarrow GL_2$. We may therefore view the function $W(g)W_{\theta^{\psi}}(\mathfrak{s}(g)) f_{2s-1}(\mathfrak{s}(g))$ as a function on $GL_2$. For convenience, we will omit the section $\mathfrak{s}$  by letting $g$ denote $\mathfrak{s}(g)$ for $g \in GL_2$ throughout Section \ref{sec3} and Section \ref{sec4}.

\section{Exceptional Poles and Derivatives}
\label{sec3}

\subsection{Derivatives and Whittaker models}
\label{sec3:1}

We establish Whittaker models for derivatives of exceptional representations. A topological group is called an $\ell$-\textit{group} if it is a Hausdorff space and has a neighborhood base at the identity consisting of compact open subgroups. If $G$ is any $\ell$-group, we denote by $\mathrm{Rep}(G)$ the category of smooth complex $G$-modules. Let $P_1$ denote  the identity matrix $I_2$. Kable \cite{Kable01} manufactures four functors
\begin{align*} 
  \Phi^- : \mathrm{Rep}(\widetilde{P}) \rightarrow \mathrm{Rep}(\widetilde{P}_1)   & &   \Psi^- : \mathrm{Rep}(\widetilde{P}) \rightarrow \mathrm{Rep}(\widetilde{GL}_1) \\
   \Phi^+ : \mathrm{Rep}(\widetilde{P}_1) \rightarrow \mathrm{Rep}(\widetilde{P}) &   & \Psi^+ : \mathrm{Rep}(\widetilde{GL}_1) \rightarrow \mathrm{Rep}(\widetilde{P}) 
\end{align*}
which he attributes to Bernstein and Zelevinsky \cite{BernsteinZelevinsky}. For $\tau \in \mathrm{Rep}(\widetilde{P})$,  $\Phi^-(\tau)=\tau \slash \tau(N^{\ast},\psi)$ is the twisted Jacquet functor or the twisted localization functor, where $\tau(N^{\ast},\psi)=\langle \tau(\mathfrak{s}(n))v-\psi(n)v \; | \; n \in N, \; v \in \tau \rangle$ and $\Psi^-(\tau)=\tau \slash \tau(N^{\ast},\textbf{1})$ is the Jacquet functor, where $\tau(N^{\ast},\textbf{1})=\langle \tau(\mathfrak{s}(n))v-v \; | \; n \in N, \; v \in \tau \rangle$. It is crucial to note that actions of the groups $\widetilde{P}_1$ and $\widetilde{GL}_1$ on $\Phi^-(\tau)$ and $\Psi^-(\tau)$ are normalized by a suitable modulus character $|\mathrm{det}|^{-\frac{1}{2}}$. For $\textbf{1} \in \mathrm{Rep}(\widetilde{P}_1)$, we put $\Phi^+(\textbf{1})=\mathrm{c\text{-}ind}^{\widetilde{P}_2}_{\widetilde{P}_1N^{\ast}}((\textbf{1} \boxtimes \psi)\otimes |\mathrm{det}|^{\frac{1}{2}})$, where the induction is an unnormalized compactly induced induction. For $\sigma \in \mathrm{Rep}(\widetilde{GL}_1)$, the induction functor $\Psi^+$ is given by $\Psi^+(\sigma)=\mathrm{ind}_{\widetilde{GL}_1N^{\ast}}^{\widetilde{P}}((\sigma \boxtimes \textbf{1})\otimes |\mathrm{det}|^{\frac{1}{2}})$. We emphasize that all four functors take genuine representations into genuine representations.

\par
For our purpose, we reconstruct Whittaker models for the first derivative of exceptional representations \cite[\S 3]{Gelbart-PS} in the context of Cogdell and Piatetski-Shapiro \cite[\S 1]{Cogdell-PS}. Let us mention that what Gelbart and Piatetski-Shapiro denoted by $r_{\chi}$ with $\chi=\textbf{1}_{F^{\times}}$ a trivial character in \cite[Proposition 2.3.3]{Gelbart-PS} corresponds to what we mean by $\theta^{\psi}$ in this paper. Even if $(\theta^{\psi},V_{\theta^{\psi}})$ is irreducible, $\theta^{\psi}|_{\widetilde{P}}$ will not be irreducible and so we shall be forced to build a natural filtration by $\widetilde{P}$-submodules
\[
  \{ 0 \} \subset \tau_2 \subset \tau_1=\theta^{\psi}|_{\widetilde{P}}
\]
such that $\tau_2 \simeq \mathrm{c\text{-}ind}^{\widetilde{P}_2}_{\widetilde{P}_1N^{\ast}}(\textbf{1} \boxtimes \psi) \simeq \Phi^+(\tau_{(1)})$ and $\tau_1 \slash \tau_2 \simeq \Psi^+({\theta^{\psi}}^{(1)})$, where the representations $\tau_{(1)} \in \mathrm{Rep}(\widetilde{P}_1)$ and ${\theta^{\psi}}^{(1)} \in \mathrm{Rep}(\widetilde{GL}_1)$ are defined by
$\tau_{(1)}=\Phi^-(\theta^{\psi}|_{\widetilde{P}})$ and ${\theta^{\psi}}^{(1)}=\Psi^-(\theta^{\psi}|_{\widetilde{P}})$. The proof of these statements can be found in the work of Bernstein and Zelevinsky \citelist{\cite{BernsteinZelevinsky}\cite{Wang}*{(2.3)}}.

\par
We consider the Kirillov map
\begin{equation}
\label{Kirillov}
  v \mapsto (W_{\theta^{\psi}})_v \left( \mathfrak{s} \begin{pmatrix} a & \\ & 1 \end{pmatrix} \right)
\end{equation}
which takes $V_{\theta^{\psi}}$ to a space of complex valued functions on $F^{\times}$. Furthermore each map $a \mapsto (W_{\theta^{\psi}})_v\left( \mathfrak{s} \begin{pmatrix} a & \\ & 1 \end{pmatrix} \right)$ is locally constant on $F^{\times}$ and compactly supported in $F$. According to \cite[\S 3.1]{Gelbart-PS}, the Kirillov map \eqref{Kirillov} is injective, that is, $(W_{\theta^{\psi}})_v\left( \mathfrak{s} \begin{pmatrix} a & \\ & 1 \end{pmatrix} \right)= 0$ implies that $v=0$. Therefore the representation of $\widetilde{P}$ has a  realization on the space of functions on $F^{\times}$, which is called the \textit{Kirillov model}, $K(\theta^{\psi},\psi^{-1})$.

\begin{proposition}  \cite[\S 3.1 (Gelbart and Piatetski-Shapiro)]{Gelbart-PS}
\label{GPS}
Let $\theta^{\psi}_{(0)}=\theta^{\psi}|_{\widetilde{P}}$. Then in terms of Whittaker model $\theta^{\psi}_{(0)}$, we have
\[
  K(\theta^{\psi},\psi^{-1}):=\mathcal{W}(\theta^{\psi}_{(0)},\psi^{-1})=\left\{ W : \mathfrak{s} \begin{pmatrix} a & \\ & 1 \end{pmatrix} \mapsto W \left( \mathfrak{s} \begin{pmatrix} a & \\ & 1 \end{pmatrix} \right)\; \middle|\; W \in  \mathcal{W}(\theta^{\psi},\psi^{-1}), a \in F^{\times} \right\}
\]
endowed with the natural action of $\widetilde{P}$ by right translation.
\end{proposition}

We analyze the behavior of the Kirillov function $ W_{\theta^{\psi}} \left( \mathfrak{s} \begin{pmatrix} a & \\ & 1 \end{pmatrix} \right)$ near $0$.
 Let $\mathcal{S}(F^{\times})$ be the space of locally constant functions $\varphi : F^{\times} \rightarrow \mathbb{C}$ with compact supports in $F^{\times}$. The representation $R_{\theta^{\psi}}$ of $\widetilde{B}$ on the space $\mathcal{S}(F^{\times})$ is given by
 \[
   R_{\theta^{\psi}} \begin{pmatrix} a & x \\ & 1 \end{pmatrix} \varphi(y)=\psi(xy)\varphi(ya) \quad \text{and} \quad
   R_{\theta^{\psi}} \begin{pmatrix} z &  \\ & z \end{pmatrix} \varphi(y)=(y,z)_F\mu_{\psi}(z)^{-1}\varphi(y).
 \]
Then $\mathcal{S}(F^{\times})$ is nothing but the model of $\mathcal{W}(\Phi^+(\tau_{(1)}),\psi^{-1})$ under the map $\varphi \in \mathcal{S}(F^{\times}) \mapsto W \in \mathcal{W}(\Phi^+(\tau_{(1)}),\psi^{-1})$, where $\varphi(y)=W \left( \mathfrak{s} \begin{pmatrix} y & \\ & 1 \end{pmatrix} \right)$.

\begin{proposition} \cite[\S 3.2 (Cogdell, Gelbart, and Piatetski-Shapiro)]{Gelbart-PS} 
\label{GPS2}
Let $\theta^{\psi}_{(0)}=\theta^{\psi}|_{\widetilde{P}}$ and $\theta^{\psi}_{(1)}=\Phi^-(\theta^{\psi}|_{\widetilde{P}})$. 
As a $\widetilde{P}$-modules, $\Phi^+(\theta^{\psi}_{(1)})$ has a model of space of functions
\[
\begin{split}
 \mathcal{W}(\Phi^+(\theta^{\psi}_{(1)}),\psi^{-1}) =&\left\{ W : \mathfrak{s} \begin{pmatrix} a & \\ & 1 \end{pmatrix} \mapsto W \left( \mathfrak{s} \begin{pmatrix} a & \\ & 1 \end{pmatrix} \right)\; \right|\; W \in  \mathcal{W}(\theta^{\psi},\psi^{-1}), a \in F^{\times} \mathrm{\;and\; there} \\
  &\left. \quad \mathrm{exists}\; N>0\; \mathrm{such\;that\;} W\left( \mathfrak{s} \begin{pmatrix} a & \\ & 1 \end{pmatrix}\right) = 0 \;\mathrm{whenever}\; |a|<q^{-N} \right\} \simeq \mathcal{S}(F^{\times})
 \end{split}
  \]
  that is isomorphic to $\mathcal{W}(\theta^{\psi}_{(0)},\psi^{-1})(N^{\ast},\mathbf{1})=\langle \theta^{\psi}(\mathfrak{s}(n))W-W \; | \; W \in \mathcal{W}(\theta^{\psi},\psi^{-1}), n \in N \rangle$.
\end{proposition}

As in \cite[\S 3.2]{Gelbart-PS}, $\mathcal{W}(\Phi^+(\theta^{\psi}_{(1)}),\psi^{-1})$ is the Whittaker model which Gelbart and Piatetski-Shapiro denote by $K_0(\pi)$ and $ K(\theta^{\psi},\psi^{-1}) \slash \mathcal{S}(F^{\times}) \simeq  \mathcal{W}(\theta^{\psi}_{(0)},\psi^{-1}) \slash  \mathcal{W}(\theta^{\psi}_{(0)},\psi^{-1})(N^{\ast},\mathbf{1})$ is the Jacquet module which they denote by $J(\pi)$. Before we turn to the Rankin-Selberg integrals, we illustrate the connection between the Whittaker model for $\theta^{\psi}$ and the first derivative ${\theta^{\psi}}^{(1)}$. For every $b \in F^{\times}$, we define a quadratic character $\chi_b$ of $F^{\times}$ by $\chi_b(a)=(b,a)_F$. Let $\nu(g)=|\mathrm{det}(g)|$ denote the unramified determinant character of $GL_1$ or $GL_2$. In virtue of Kable \cite[Theorem 5.2]{Kable01}, let us write 
\begin{equation}
\label{first-decomp}
{\theta^{\psi}}^{(1)} \otimes \nu^{\frac{1}{4}} \simeq \underset{{b \in (F^{\times})^2 \backslash F^{\times}}}{\bigoplus} \chi_b.
\end{equation}
The character $\chi_b$ is an irreducible subrepresentation of the representation ${\theta^{\psi}}^{(1)}$ with the normalized quotient model $\mathcal{W}(\theta^{\psi}_{(0)},\psi^{-1}) \slash  \mathcal{W}(\theta^{\psi}_{(0)},\psi^{-1})(N^{\ast},\mathbf{1})$. We denote $\theta^{\psi}_{(0),b}$ by the inverse image of $\chi_b$ in $\theta^{\psi}_{(0)}$ for the canonical normalized projection from $\theta^{\psi}_{(0)}$ onto ${\theta^{\psi}}^{(1)}$  and by $\mathcal{W}(\theta^{\psi}_{(0),b},\psi^{-1})$ the corresponding subspace of $\mathcal{W}(\theta^{\psi}_{(0)},\psi^{-1})$. Let $\mathcal{S}(F)$ denote the space of complex-valued locally constant functions with compact supports in $F$. We reformulate \cite[Proposition 3.4]{Gelbart-PS} and \cite[Remark 3.3.7]{Gelbart-PS} in the framework of the Cogdell and Piatetski-Shapiro interpretation of derivatives \cite[Corollary to Proposition 1.7]{Cogdell-PS}.

\begin{theorem}[Cogdell, Gelbart, and Piatetski-Shapiro]
\label{connection}
With notations as above, we have the following: for every $W_{\theta^{\psi}_{(0),b}} \in \mathcal{W}(\theta^{\psi}_{(0),b},\psi^{-1})$ and for all $\varphi_{\circ}$ in $ \mathcal{S}(F)$ locally constant and supported in a sufficient small neighborhood of 0, there exists a character $\chi_b$ and a constant $c \in \mathbb{C}$ such that
\[
  W_{\theta^{\psi}_{(0),b}} \left( \mathfrak{s}\begin{pmatrix} a& \\ & 1 \end{pmatrix} \right) \varphi_{\circ}(a)=
  \begin{cases}
  c|a|^{\frac{1}{4}}  \chi_b(a) \varphi_{\circ}(a), & a \in (F^{\times})^2\cdot b \\
  0, & a \not \in (F^{\times})^2\cdot b.
\end{cases}
\]
Conversely, for any character $\chi_b$ there exist $W_{\theta^{\psi}_{(0),b}} \in \mathcal{W}(\theta^{\psi}_{(0),b},\psi^{-1})$ and $\varphi_{\circ} \in \mathcal{S}(F)$ non-vanishing at zero such that
\[
   \begin{rcases}
 \chi_b(a) \varphi_{\circ}(a), &a \in (F^{\times})^2\cdot b  \\
  0, & a \not \in (F^{\times})^2\cdot b
\end{rcases}
= |a|^{-\frac{1}{4}}W_{\theta^{\psi}_{(0),b}}  \left( \mathfrak{s}\begin{pmatrix} a& \\ & 1 \end{pmatrix} \right)\varphi_{\circ}(a).
\]
\end{theorem}
The basic properties of Section \ref{sec3:1} remain unchanged in the usual $GL_2$-setting. The best reference for the theory of derivatives and Whittaker functions on $GL_n$ is \cite{Cogdell-PS}.

\subsection{Exceptional poles and distinctions}
\label{sec3:2}
In \S \ref{sec3:2} we investigate the exceptional and regular Zeta integral. For each $W \in \mathcal{W}(\pi,\psi)$, $W_{\theta^{\psi}} \in \mathcal{W}(\theta^{\psi},\psi^{-1})$, and $f_{2s-1} \in V_{hol}(2s-1,\omega^{-1}_{\pi})$, we define the regular integral
\[
  I_{reg}(W,W_{\theta^{\psi}},f_{2s-1})=\int_{Z^2N \backslash GL_2}  W(g)W_{\theta^{\psi}}(g) f_{2s-1}(g) dg
\]
which converges for $\mathrm{Re}(s)$ large and extends to $\mathbb{C}$ as a function of $\mathbb{C}(q^{-s/2})$. We will use the following property that holomorphic sections are $GL_2$-stable to define a $\mathbb{C}[q^{\pm s/2}]$-fractional ideal.

\begin{lemma}
For $f_{2s-1} \in V_{hol}(2s-1,\eta)$ and $g \in GL_2$, $R(\mathfrak{s}(g))f_{2s-1}$ is a holomorphic section.
\end{lemma}

\begin{proof}
The holomorphic sections are slightly different from those treated in \cite[Lemma 4.1]{JO20}. Observing that $(F^{\times})^2 \backslash F^{\times}$ is finite, the argument \citelist{\cite{JO20}*{Lemma 4.1}\cite{Kaplan13}*{Claim 2.1}} can easily be modified to deal with our corresponding holomorphic sections. 
\end{proof}

The $\mathbb{C}$-vector space
\[
 \mathcal{I}_{reg}(\pi)=\langle I_{reg}(W,W_{\theta^{\psi}},f_{2s-1}) \;|\; W \in \mathcal{W}(\pi,\psi), W_{\theta^{\psi}} \in \mathcal{W}(\theta^{\psi}), 
 f_{2s-1} \in V_{hol}(2s-1,\omega^{-1}_{\pi}) \rangle
\]
spanned by the integrals $I_{reg}(W,W_{\theta^{\psi}},f_{2s-1})$ is a $\mathbb{C}[q^{\pm s/2}]$-fractional ideal of $\mathbb{C}(q^{-s/2})$. For $W \in \mathcal{W}(\pi,\psi)$, $W_{\theta^{\psi}} \in \mathcal{W}(\theta^{\psi},\psi^{-1})$ and $s \in \mathbb{C}$, we define the simple zeta integral
\[
  I_{(0)}(s;W,W_{\theta^{\psi}})=\int_{F^{\times}}  W \begin{pmatrix} a & \\ & 1 \end{pmatrix} W_{\theta^{\psi}} \begin{pmatrix} a & \\ & 1 \end{pmatrix}  |a|^{\frac{s}{2}-\frac{3}{4}} d^{\times} a.
\]
$I_{(0)}(s,W,W_{\theta^{\psi}})$ converges to an element of $\mathbb{C}(q^{-s/2})$ for $\mathrm{Re}(s)$ large enough. Let $\mathcal{I}_{(0)}(\pi)$ denote a $\mathbb{C}$-vector space of $\mathbb{C}(q^{-s/2})$ spanned by $I_{(0)}(s,W,W_{\theta^{\psi}})$ as $W$ is taken over $\mathcal{W}(\pi,\psi)$ and $W_{\theta^{\psi}}$ runs through $\mathcal{W}(\theta^{\psi},\psi^{-1})$. Then $\mathcal{I}_{(0)}(\pi)$ is a $\mathbb{C}[q^{\pm s/2}]$-fractional ideal which is related to $\mathcal{I}_{reg}(\pi)$ by Proposition below. For any open compact set $K^{\circ}$ of $K$ and any element $k$ of $K$, we first define the test function $\mathbbm{1}_{kK^{\circ},2s-1} \in I(2s-1,\eta)$ by
\begin{equation}
\label{char}
\begin{split}
& \mathbbm{1}_{kK^{\circ},2s-1}(\widetilde{g}) 
=
 \begin{cases}
  \xi\eta(z)  \delta_B(p)^{(2s+1)/4} & \text{if}\;\;  \widetilde{g}=(1,\xi)\mathfrak{s}(z) \mathfrak{s}(p) \mathfrak{s}(k) \mathfrak{s}(k^{\circ}), z \in Z^2, p \in P, k^{\circ} \in K^{\circ} \\
  0 &  \text{if} \;\;  \widetilde{g} \notin \widetilde{Z}^2\widetilde{P} \mathfrak{s}(k) \mathfrak{s}(K^{\circ}).
 \end{cases}\\
 \end{split}
\end{equation}

\begin{proposition}
As $\mathbb{C}[q^{s/2},q^{-s/2}]$-fractional ideals, we have $\mathcal{I}_{(0)}(\pi)=\mathcal{I}_{reg}(\pi)$. Furthermore we obtain the inclusion 
$\mathcal{I}_{(0)}(\pi) \subset  \mathcal{I}(\pi)$ of $\mathbb{C}[q^{s/2},q^{-s/2}]$-fractional ideals.
\end{proposition}

\begin{proof}
Let $f_{2s-1} \in V_{hol}(2s-1,\omega^{-1}_{\pi})$ be given. We write $f_{2s-1}=\sum_{i=1}^mP_i(q^{s/2},q^{-s/2})f_{2s-1}^{(i)}$ with $P_i(X) \in \mathbb{C}[X]$
and $f^{(i)}_{2s-1} \in V_{std}(2s-1,\omega_{\pi}^{-1})$. We let $K_{\circ} \subset K$ be a compact open subgroup which stabilizes each $f^{(i)}_{2s-1}$  as well as  $W$ and $W_{\theta^{\psi}}$. Write $K=\cup_j k_jK_{\circ}$. Since the integrand is left $N$-invariant and right $K_{\circ}$-invariant, we reach
\[
  I_{reg}(W,W_{\theta^{\psi}},f_{2s-1})
 =c_1\sum_{i,j} P_i(q^{\pm s/2}) \int_{Z^2 \backslash T} [\pi(k_j)W](t) [\theta^{\psi}(k_j)W_{\theta^{\psi}}](t) [R(k_j)f^{(i)}_{2s-1}](t) \delta_{B}(t)^{-1} dt
\]
by the Iwasawa decomposition. Here $c_1$ is a volume of $K_{\circ}$. $Z^2$ has a finite index in $Z$. Taking the product $T=ZA$ into account, the integral $I_{reg}(W,W_{\theta^{\psi}},f_{2s-1})$ can be decomposed as a finite sum of the form
\[
\begin{split}
  I_{reg}(W,W_{\theta^{\psi}},f_{2s-1})
  &= c_1\sum_{i,j} \sum_{(F^{\times})^2 \backslash F^{\times}} P_i(q^{s/2},q^{-s/2}) \omega_{\pi}(b) \mu_{\psi}(b)^{-1} f^{(i)}_{2s-1}(t(b,b)k_j) 
     \\
    &\quad \times\int_{F^{\times}} [\pi(k_j)W] \begin{pmatrix} a & \\ & 1 \end{pmatrix} [\theta^{\psi}(k_j)W_{\theta^{\psi}}] \begin{pmatrix} a & \\ & 1 \end{pmatrix} \chi^{-1}_b(a)
    |a|^{\frac{s}{2}-\frac{3}{4}}d^{\times} a
    \\
  &=\sum_{i,j} \sum_{(F^{\times})^2 \backslash F^{\times}} Q_{i}(q^{s/2},q^{-s/2};b,k_j)  I_{(0)}(s,\pi(k_j)W,\theta^{\psi}(k_j)W_{\theta^{\psi}})
\end{split}
\]
with $Q_{i}(X;b,k_j) \in \mathbb{C}[X]$. The second equality follows from \cite[Lemma 1.9 (3)]{Ya17}) that $\theta^{\psi} \otimes \chi^{-1}_b \simeq \theta^{\psi}$ for a quadratic character $\chi^{-1}_b$. This confirms that $\mathcal{I}_{reg}(\pi) \subset  \mathcal{I}_{(0)}(\pi)$.

\par
For the reverse inclusion, we choose $K^{\prime}$ for $W \in \mathcal{W}(\pi,\psi)$ and $W_{\theta^{\psi}} \in \mathcal{W}(\theta^{\psi},\psi^{-1})$ to be right invariant under $K^{\prime}$. We then take $f_{2s-1}$ a characteristic function of $ \mathbbm{1}_{K^{\prime},2s-1}$ in \eqref{char}. The integral $I_{reg}(W,W_{\theta^{\psi}},f_{2s-1})$ reduces to
\[
  I_{reg}(W,W_{\theta^{\psi}},f_{2s-1})=c_2\int_{F^{\times}}  W \begin{pmatrix} a & \\ & 1 \end{pmatrix} W_{\theta^{\psi}} \begin{pmatrix} a & \\ & 1 \end{pmatrix} |a|^{\frac{s}{2}-\frac{3}{4}} d^{\times} a=c_2I_{(0)}(s,W,W_{\theta^{\psi}}),
\]
where $c_2$ is a volume of $K^{\prime}$. Therefore we obtain the desired inclusion $\mathcal{I}_{reg}(\pi) \supset  \mathcal{I}_{(0)}(\pi)$.
\end{proof}

According to Theorem 7.1 of \cite{BuGi} or Proposition 3.8 (4) of \cite{Ya17}, we have the following non-vanishing result.

\begin{proposition}
There exist $W \in \mathcal{W}(\pi,\psi)$ and $W_{\theta^{\psi}} \in \mathcal{W}(\theta^{\psi},\psi^{-1})$ such that $I_{(0)}(s,W,W_{\theta^{\psi}})=1$.
\end{proposition}

As a consequence, $ \mathcal{I}_{(0)}(\pi)$ contains $1$. We can find a polynomial $P(X) \in \mathbb{C}[X]$ such that $P(0)=1$ and $1/P(q^{-s/2})$ generates the fractional ideal $\mathcal{I}_{reg}(\pi)=\mathcal{I}_{(0)}(\pi)$. We denote by the \textit{regular} $L$-function
\[
  L_{reg}(s,\pi,\mathrm{Sym}^2)=L_{(0)}(s,\pi,\mathrm{Sym}^2)=\frac{1}{P(q^{-s/2})}.
\]
We assume that 
$
\dfrac{I(W,W_{\theta^{\psi}},f_{2s-1})}{ L_{reg}(s,\pi,\mathrm{Sym}^2)} 
$
has a pole for some $W \in \mathcal{W}(\pi,\psi)$, $W_{\theta^{\psi}} \in \mathcal{W}(\theta^{\psi},\psi^{-1})$, and $f_{2s-1} \in V_{good}(2s-1,\omega_{\pi})$. Such poles are called {\it exceptional poles}. Then the poles of the ratio $L(s,\pi,\mathrm{Sym}^2)/L_{reg}(s,\pi,\mathrm{Sym}^2)$ are exactly exceptional poles of $L(s,\pi,\mathrm{Sym}^2)$. We now show that poles of this fraction are all simple.

\begin{proposition}
\label{Exceptional pole}
The ratio $\dfrac{L(s,\pi,\mathrm{Sym}^2)}{L_{reg}(s,\pi,\mathrm{Sym}^2)}$ has simple poles and in particular $s=s_0$ is the pole of $L(2s,\omega_{\pi}^2)$ if $s=s_0$ is an exceptional pole of $L(s,\pi,\mathrm{Sym}^2)$.
\end{proposition}

Before proceeding the proof, we would like to mention the analytic property of the intertwining operator $M(2s-1,\omega^{-1}_{\pi})$ in \citelist{\cite{Takeda15}*{Lemma 4.3}\cite{Ya17}*{Lemma 3.2}}.

\begin{lemma}[Takeda]
\label{Takeda}
The operator 
\[
  \frac{1}{L(2s-1,\omega^2_{\pi})}M(2s-1,\omega^{-1}_{\pi}) : I(2s-1,\omega^{-1}_{\pi}) \rightarrow J(1-2s,\omega^{-1}_{\pi})
\]
is holomorphic for all $s \in \mathbb{C}$.
\end{lemma}

\begin{proof}[Proof of Proposition \ref{Exceptional pole}]
For any $f_{2s-1} \in V_{good}(2s-1,\omega^{-1}_{\pi})$, Lemma \ref{Takeda} is exploited to take $P_i(q^{\pm s/2}) \in \mathbb{C}[q^{\pm s/2}]$ and $f^{(i)}_{2s-1} \in V_{std}(2s-1,\omega^{-1}_{\pi})$ such that $I(W,W_{\theta^{\psi}},f_{2s-1})$ is a finite sum of the form 
\[
I(W,W_{\theta^{\psi}},f_{2s-1})=L(2s,\omega_{\pi}^2)\sum_{i=1}^mP_i(q^{ \pm s/2})I(W,W_{\theta^{\psi}},f^{(i)}_{2s-1}).
\]
Indeed the fraction $\dfrac{I(W,W_{\theta^{\psi}},f^{(i)}_{2s-1})}{L_{reg}(s,\pi,\mathrm{Sym}^2)}$ is entire. This implies that $L(s,\pi,\mathrm{Sym}^2)^{-1}$ divides a product $L(2s,\omega_{\pi}^2)^{-1}L_{reg}(s,\pi,\mathrm{Sym}^2)^{-1}$ in $\mathbb{C}[q^{\pm s/2}]$. Thus the poles of the ratio
$\dfrac{L(s,\pi,\mathrm{Sym}^2)}{L_{reg}(s,\pi,\mathrm{Sym}^2)}$ are found among the poles of $L(2s,\omega_{\pi}^2)$.
\end{proof}

In the spirit of \cite{Schmidt-Tran}, we define \textit{exceptional} $L$-functions in such a way that exceptional poles can be regular but $L(s,\pi,\mathrm{Sym}^2) \neq L_{reg}(s,\pi,\mathrm{Sym}^2)$ precisely if there exist exceptional poles.

\begin{definition}
We let 
\[
 L_{ex}(s,\pi,\mathrm{Sym}^2)=\prod_{s_0} (1-q^{s_0/2}q^{-s/2})^{-1}
\]
where $s_0$ is taken over all exceptional poles of $L_{ex}(s,\pi,\mathrm{Sym}^2)$. 
\end{definition}

With this said, we can factor $L(s,\pi,\mathrm{Sym}^2)$ as
\begin{equation}
\label{product}
L(s,\pi,\mathrm{Sym}^2)=L_{ex}(s,\pi,\mathrm{Sym}^2)L_{reg}(s,\pi,\mathrm{Sym}^2). 
\end{equation}
When it is clear what $\omega_{\pi}$ we are working with, we will abuse the notation by letting $\omega_{\pi}=\omega_{\pi} \circ \mathrm{det}$. Now we would like to characterize the occurrence of exceptional poles. To this end, we need the following lemma \cite[Lemma 1.15]{Ya17}.

\begin{lemma}[Yamana]
\label{image-Yamana}
Let $\omega_{\pi}$ be a quadratic character of $F^{\times}$. The representation $I(1,\omega^{-1}_{\pi})$ has the unique irreducible quotient, which is isomorphic to $\theta^{\psi^{-1}}\otimes \omega_{\pi}$. Furthermore the quotient map
$I(1,\omega^{-1}_{\pi}) \rightarrow \theta^{\psi^{-1}} \otimes \omega_{\pi}$ is realized as the intertwining operator $M(1,\omega^{-1}_{\pi})$.
\end{lemma}

It is worthwhile noting from \citelist{\cite{Gelbart-PS}*{\S 1.3}\cite{Ya17}*{Remark 1.7}} that $\theta^{{\psi}^{-1}}$ is independent of $\psi$ hence we may suppress the superscript $\psi$. 
We say that an irreducible admissible representation $(\pi,V_{\pi})$ of $GL_2$ is {\it $\theta$-distinguished} if $\pi \otimes \theta^{\psi} \otimes \theta^{\psi^{-1}}$ admits a nonzero $GL_2$-invariant trilinear form. A $\theta$-distinguished representation is characterized in terms of poles of symmetric square $L$-functions for self-dual representations at $s=0$.

\begin{proposition}\cite[Corollary 3.9]{Ya17}
\label{char-Yamana}
We assume that $\mathrm{dim}_{\mathbb{C}} \mathrm{Hom}_{GL_2}(\pi \otimes\theta^{\psi} \otimes I(1,\omega^{-1}_{\pi}), \mathbb{C} ) \leq 1$ and $\omega_{\pi}$ is a quadratic character. Then the following conditions are equivalent:
\begin{enumerate}[label=$(\roman*)$]
\item $\pi \otimes \omega_{\pi}$ is $\theta$-distinguished.
\item $\mathrm{Hom}_{GL_2}(\pi \otimes\theta^{\psi} \otimes \theta^{\psi^{-1}} \otimes \omega_{\pi}, \mathbb{C} )\simeq \mathrm{Hom}_{GL_2}(\pi \otimes\theta^{\psi} \otimes I(1,\omega^{-1}_{\pi}), \mathbb{C} ) \neq 0$
\item There is a non-zero functional $\Lambda : \pi \otimes\theta^{\psi} \otimes I(1,\omega^{-1}_{\pi}) \rightarrow \mathbb{C}$ which factors through the quotient map $ \pi \otimes\theta^{\psi} \otimes I(1,\omega^{-1}_{\pi})  \rightarrow \pi \otimes\theta^{\psi} \otimes \theta^{\psi^{-1}}\otimes \omega_{\pi}$.
\end{enumerate}
\end{proposition}

According to \cite[Theorem 2.14]{Ya17}, one dimensionality of the space $\mathrm{Hom}_{GL_2}(\pi \otimes\theta^{\psi} \otimes I(1,\omega^{-1}_{\pi}), \mathbb{C} )$ is available for irreducible unitary representations of $GL_2$. Though this might be hold for all irreducible admissible representations of $GL_2$, the author does not know if the same technique in the proof of \cite[Theorem 2.14]{Ya17} applies to these representations at this moment. To overcome this,
we exploit the deformation of representations in $\S$\ref{sec4}. We finally come to give a characterization of exceptional poles.

\begin{theorem}
\label{main-characterization}
Let $\pi$ be an irreducible admissible generic representation of $GL_2$. Suppose that $\mathrm{dim}_{\mathbb{C}} \mathrm{Hom}_{GL_2}(\pi \otimes\theta^{\psi} \otimes I(1,\omega^{-1}_{\pi}), \mathbb{C} ) \leq 1$. If $L_{ex}(s,\pi,\mathrm{Sym}^2)$ has a pole at $s=0$, then $\pi$ is $\theta$-distinguished.
\end{theorem}

\begin{proof}
The proof is inspired by that of \cite[Theorem 4.7]{JO20}. We provide the complete detail to make thing concrete. The fraction
\[
 \frac{I(W,W_{\theta^{\psi}},f_{2s-1})}{L(s,\pi,\mathrm{Sym}^2)}=\frac{I(W,W_{\theta^{\psi}},f_{2s-1})}{L_{ex}(s,\pi,\mathrm{Sym}^2)L_{reg}(s,\pi,\mathrm{Sym}^2)}
\]
is entire in $s$, so we can evaluate at $s=0$. For $s=0$, the Zeta integral defines an intertwining map $T_0 : \mathcal{W}(\pi,\psi) \times \mathcal{W}(\theta^{\psi},\psi^{-1}) \times V_{good}(2s-1,\omega^{-1}_{\pi}) \rightarrow \mathbb{C}$,
\[
  (W,W_{\theta^{\psi}},f_{2s-1}) \mapsto  \left.\frac{I(W,W_{\theta^{\psi}},f_{2s-1})}{L(s,\pi,\mathrm{Sym}^2)}\right|_{s=0}.
\]
Suppose that $L_{ex}(s,\pi,\mathrm{Sym}^2)$ has a pole at $s=0$. Then the family $I(W,W_{\theta^{\psi}},f_{2s-1})$ for good sections $f_{2s-1} \in V_{good}(2s-1,\omega^{-1}_{\pi})$ has a higher order pole at $s=0$ than the family $I(W,W_{\theta^{\psi}},f_{2s-1})$ for holomorphic sections $f_{2s-1} \in V_{hol}(2s-1,\omega^{-1}_{\pi})$ so $T_0(W,W_{\theta^{\psi}},f_{2s-1})$ is identically zero for $f_{2s-1} \in V_{hol}(2s-1,\omega^{-1}_{\pi})$. It follows from Definition \ref{symsquare-dfn} of $L$-functions that $T_0$ is a non-zero element. Hence we choose an element $h_{2s-1} \in V_{good}(2s-1,\omega^{-1}_{\pi})$ which is not in holomorphic sections and that 
\[
 \left.\frac{I(W,W_{\theta^{\psi}},h_{2s-1})}{L(s,\pi,\mathrm{Sym}^2)}\right|_{s=0} \quad \text{is non zero.}
\]
Then $h_{2s-1}=N(1-2s,\omega_{\pi},\psi^{-1})f_{1-2s}$ for some $f_{1-2s} \in V_{hol}(1-2s,\omega_{\pi})$. Since the exceptional $L$-function $L_{ex}(s,\pi,\mathrm{Sym}^2)^{-1}$ divides $L(2s,\omega_{\pi}^2)^{-1}$ in $\mathbb{C}[q^{\pm s/2}]$, $\omega_{\pi}$ is the non-trivial quadratic chacter of $F^{\times}$. Thus we obtain
\[
\begin{split}
  0 \neq  \left.\frac{I(W,W_{\theta^{\psi}},h_{2s-1})}{L(s,\pi,\mathrm{Sym}^2)}\right|_{s=0}
  &=\left.\frac{I(W,W_{\theta^{\psi}},N(1-2s,\omega_{\pi},\psi^{-1})f_{1-2s})}{L(s,\pi,\mathrm{Sym}^2)}\right|_{s=0}\\
  &=c_3\left.\frac{\varepsilon(1-2s,\textbf{1}_{F^{\times}},\psi^{-1})}{L(1-2s,\textbf{1}_{F^{\times}})}\cdot \frac{I(W,W_{\theta^{\psi}},{^{\iota}[{\hat{N}}(1-2s,\omega_{\pi},\psi^{-1})f_{1-2s}]})}{L_{reg}(s,\pi,\mathrm{Sym}^2)} \right|_{s=0}
\end{split}
\]
with a non-zero constant $c_3=\dfrac{\mathrm{Res}_{s=0}L(2s,\textbf{1}_{F^{\times}})}{\mathrm{Res}_{s=0}L_{ex}(s,\pi,\mathrm{Sym}^2)} \neq 0$. Consequently, Lemma \ref{image-Yamana} implies that the above non-trivial functional
\[
(W,W_{\theta^{\psi}},f_1) \mapsto 
\left.c_3\frac{\varepsilon(1-2s,\textbf{1}_{F^{\times}},\psi^{-1})}{L(1-2s,\textbf{1}_{F^{\times}})}\cdot \frac{I(W,W_{\theta^{\psi}},{^{\iota}[{\hat{N}}(1-2s,\omega_{\pi},\psi^{-1})f_{1-2s}]})}{L_{reg}(s,\pi,\mathrm{Sym}^2)} \right|_{s=0}
\]
factors through the quotient $ \pi \otimes\theta^{\psi} \otimes I(1,\omega^{-1}_{\pi})  \rightarrow \pi \otimes\theta^{\psi} \otimes \theta^{\psi^{-1}}\otimes \omega_{\pi}$. Appealing to \cite[Lemma 1.9 (3)]{Ya17}, one gets $\theta^{\psi^{-1}} \otimes \omega_{\pi} \simeq \theta^{\psi^{-1}}$. We conclude from Proposition \ref{char-Yamana} that $\pi$ is $\theta$-distinguished.
\end{proof}

For an exceptional pole at $s=s_0$, Theorem \ref{main-characterization} becomes the following.

\begin{corollary}
\label{general-characterization}
Let $\pi$ be an irreducible admissible generic representation of $GL_2$. If $s=s_0$ is a pole of $L_{ex}(s,\pi,\mathrm{Sym}^2)$ and $\mathrm{dim}_{\mathbb{C}} \mathrm{Hom}_{GL_2}(\pi\nu^{\frac{s_0}{2}} \otimes\theta^{\psi} \otimes I(1,\omega^{-1}_{\pi\nu^{s_0/2}}), \mathbb{C} ) \leq 1$, then $\pi\nu^{\frac{s_0}{2}}$ is $\theta$-distinguished.
\end{corollary}

\begin{proof}
We know from \cite[Remark 3.13]{Ya17} that $V_{good}(2s-1,\omega^{-1}_{\pi\nu^{s_0/2}}) \otimes \nu^{\frac{s_0}{2}} \simeq V_{good}(2(s+s_0)-1,\omega^{-1}_{\pi})$ and $V_{hol}(2s-1,\omega^{-1}_{\pi\nu^{s_0/2}}) \otimes \nu^{\frac{s_0}{2}} \simeq V_{hol}(2(s+s_0)-1,\omega^{-1}_{\pi})$. Therefore $L_{ex}(s,\pi,\mathrm{Sym}^2)$ has a pole at $s=s_0$ if and only if $L_{ex}(s,\pi\nu^{\frac{s_0}{2}},\mathrm{Sym}^2)=L_{ex}(s+s_0,\pi,\mathrm{Sym}^2)$ has a pole at $s=0$. The result right away follows from Theorem \ref{main-characterization}. 
\end{proof}

We shall see later (Proposition \ref{converse}) a kind of the converse of Corollary \ref{general-characterization}.

\subsection{Regular $L$-functions and factorizations}
We examine the regular pole at $s=s_0$ of the family $\mathcal{I}_{(0)}(\pi)$. Let $d_{s_0}$ be its maximal order in this family. The integral $ I_{(0)}(s,W,W_{\theta^{\psi}})$ has the Laurent expansion, which in case we write
\[
 I_{(0)}(s;W,W_{\theta^{\psi}})=\frac{B_{(0),s_0}(W,W_{\theta^{\psi}})}{(q^{s/2}-q^{s_0/2})^{d_{s_0}}}+\text{higher order terms}
\]
where $W$ and $W_{\theta^{\psi}}$ belong to the space $\mathcal{W}(\pi_{(0)},\psi)$ and $\mathcal{W}(\theta_{(0)}^{\psi},\psi^{-1})$ respectively (cf. Proposition \ref{GPS}). The function $W$ of $\mathcal{W}(\Phi^+(\pi_{(1)}),\psi)$ realized as a function $W \begin{pmatrix} a & \\ & 1 \end{pmatrix}$ on $F^{\times}$ has a multiplicative support in $a$ (cf. Proposition \ref{GPS2}). Then the non-trivial bilinear form $B_{(0),s_0}$ vanishes on $\mathcal{W}(\Phi^+(\pi_{(0)}),\psi)$. As a representation of $P$, $\pi_{(0)} \slash \Phi^+(\pi_{(1)})$ is isomorphic to $\Psi^+(\pi^{(1)})$. Hence $B_{(0),s_0}$ defines a non-trivial bilinear form on $\mathcal{W}(\Psi^+(\pi^{(1)}),\psi) \times \mathcal{W}(\Phi^+(\theta^{\psi}_{(0)}),\psi^{-1})$ which is quasi-invariant with respect to the action of $\widetilde{P}$. But by \cite[\S 4, Proposition 4.3]{Kable99} (cf. \cite[Proposition 3.7]{BernsteinZelevinsky}), there is no non-trivial quasi-invariant pairing between $\Psi^+(\pi^{(1)})$ and $\Phi^+(\theta^{\psi}_{(0)})$. Therefore we may view $B_{(0),s_0}$ as a non-zero bilinear form on the space $\mathcal{W}(\Psi^+(\pi^{(1)}),\psi) \times \mathcal{W}(\Psi^+({\theta^{\psi}}^{(1)}),\psi^{-1})$. Having the description of the bilinear from $B_{(0),s_0}$ in hand, we are in a position to show the main factorization.

\begin{proposition}
\label{regular-factorization}
Let $\pi$ be an irreducible admissible generic representation of $GL_2$ such that all of its derivatives are completely reducible. Then
\[
  L_{reg}(s,\pi,\mathrm{Sym}^2)^{-1}=\underset{i}{l.c.m.}\{ L(s,\pi_i^{(1)} \times \pi_i^{(1)} )^{-1} \}
\]
where the lease common multiple is with respect to divisibility in $\mathbb{C}[q^{s/2},q^{-s/2}]$ and is taken over all irreducible constituents $\pi^{(1)}_i$ of $\pi^{(1)}$.
\end{proposition}

\begin{proof}
The proof proceeds as in those of \citelist{\cite{Cogdell-PS}*{Proposition 2.3}\cite{Matringe}*{Proposition 4.14}}. 
We take a pole $s=s_0$ of the regular $L$-function $L_{reg}(s,\pi,\mathrm{Sym}^2)$ and let $d_{s_0}$ be its order in $L_{reg}(s,\pi,\mathrm{Sym}^2)$. Then it occurs as a pole of order $d_{s_0}$ of the simple Zeta integral
\[
  I_{(0)}(s;W,W_{\theta^{\psi}})=\int_{F^{\times}} W \begin{pmatrix} a & \\ & 1 \end{pmatrix} W_{\theta^{\psi}} \begin{pmatrix} a & \\ & 1 \end{pmatrix} |a|^{\frac{s}{2}-\frac{3}{4}} \;d^{\times} a
\]
for some $W$ in $\mathcal{W}(\pi,\psi)$ and $ W_{\theta^{\psi}}$ in $\mathcal{W}(\theta^{\psi},\psi^{-1})$. Moreover for any $\varphi \in \mathcal{S}(F)$ with $\varphi(0)=1$ and $W_{\theta^{\psi}}$ in  $\mathcal{W}(\theta^{\psi},\psi^{-1})$, the integral
\[
  I^1_{(0)}(s;W,W_{\theta^{\psi}}):=\int_{F^{\times}} W \begin{pmatrix} a & \\ & 1 \end{pmatrix} W_{\theta^{\psi}} \begin{pmatrix} a & \\ & 1 \end{pmatrix}(1-\varphi(a)) |a|^{\frac{s}{2}-\frac{3}{4}} \;d^{\times} a.
\]
is always entire. Therefore $s=s_0$ is a pole of order $d_{s_0}$ of the integral
\[
  I^0_{(0)}(s;W,W_{\theta^{\psi}}):=\int_{F^{\times}} W \begin{pmatrix} a & \\ & 1 \end{pmatrix} W_{\theta^{\psi}} \begin{pmatrix} a & \\ & 1 \end{pmatrix}\varphi(a) |a|^{\frac{s}{2}-\frac{3}{4}} \;d^{\times} a.
\]
after writing $I_{(0)}(s;W,W_{\theta^{\psi}})=I^0_{(0)}(s;W,W_{\theta^{\psi}})+I^1_{(0)}(s;W,W_{\theta^{\psi}})$. Let $b_1,b_2,\dotsm, b_s$ be the distinct representatives of the left coset $(F^{\times})^2 \backslash F^{\times}$. As we have indicated in  \eqref{first-decomp}, we decompose $\pi^{(1)}$ and ${\theta^{\psi}}^{(1)}$ into direct sums of simple factors $\pi^{(1)}\simeq \oplus_i \pi^{(1)}_i$ and ${\theta^{\psi}}^{(1)}   \simeq \oplus_j (\chi_{b_j} \otimes \nu^{-\frac{1}{4}})$. Then $W$ and $W_{\theta^{\psi}}$ can be expressed as $W=W_1\oplus \dotsm \oplus W_t$ and $W_{\theta^{\psi}}=W_{b_1} \oplus \dotsm \oplus W_{b_s}$ with each $W_i$ and $W_{b_j}$ projecting on some characters $\pi^{(1)}_i$ and $\chi_{b_j} \otimes \nu^{-\frac{1}{4}}$ respectively. There exist $i$ and $j$ such that $s=s_0$ is a pole of order $d_{s_0}$ of the integral
\[
  I^0_{(0)}(s;W_i,W_{b_j})=\int_{F^{\times}} W_i \begin{pmatrix} a & \\ & 1 \end{pmatrix} W_{b_j} \begin{pmatrix} a & \\ & 1 \end{pmatrix} \varphi(a) |a|^{\frac{s}{2}-\frac{3}{4}} \;d^{\times} a.
\]
We take $\varphi$ to be a characteristic function of a small neighborhood of $0$. Upon applying \cite[Corollary to Proposition 1.7]{Cogdell-PS} accompanied with Theorem \ref{connection}, we see that
\[
\begin{split}
 &I^0_{(0)}(s;W_i,W_{b_j})\\
 &=|b_j|^{\frac{s}{2}}\int_{F^{\times}} \pi^{(1)}_i(z^2b_j) \chi_{b_j}(z^2b_j) \varphi(z^2b_j) |z|^s \;d^{\times}z=\alpha q^{-\beta s/2}\int_{F^{\times}} \pi^{(1)}_i(z)^2 \varphi(z^2b_j) |z|^s \;d^{\times}z,
\end{split}
\]
where $\alpha q^{-\beta s/2}$ is a unit in $\mathbb{C}[q^{-s/2}]$ with $\alpha \in \mathbb{C}$ and $\beta \in \mathbb{Z}$. The last equality follows from $\chi_{b_j}^2 = \textbf{1}_{F^{\times}}$. However as we further replace $\varphi(z^2b_j)$ by $\varphi(z)$, the integral simplifies to a multiple of a standard Tate integral
$ I^0_{(0)}(s,W_i,W_{b_j})=\alpha q^{-\beta s/2}I(s;\pi^{(1)}_i\times \pi^{(1)}_i,\varphi)$, where
\[
 I(s;\pi^{(1)}_i\times \pi^{(1)}_i,\varphi):=\int_{F^{\times}} \pi^{(1)}_i(z) \pi^{(1)}_i(z) \varphi(z) |z|^s \;d^{\times}z. 
 \]
 Hence $I(s;\pi^{(1)}_i\times \pi^{(1)}_i,\varphi)$ has a pole of order $d_{s_0}$ at $s=s_0$ and $L(s,\pi_i^{(1)} \times \pi_i^{(1)} )$ as well.
 
 \par
 Now let $s=s_0$ be a pole of an order $d_{s_0}$ of $L(s,\pi_i^{(1)} \times \pi_i^{(1)} )$ for some $i$. It is evident that there exist a character $\pi^{(1)}$ and $\phi$ in $\mathcal{S}(F)$ which does not vanish at zero such that the Tate integral $I(s;\pi^{(1)}_i\times \pi^{(1)}_i,\phi)$ contributes to the pole of the order $d_{s_0}$ at $s=s_0$. We observe that $\varphi$ may in fact be taken to be a characteristic function of a neighborhood of $0$ small enough because $I(s;\pi^{(1)}_i\times \pi^{(1)}_i,\phi-\varphi)$ is entire. The choice of $\varphi$ allows us to choose $\varphi(z^2)$ in place of $\varphi(z)$. We pick the identity $b_j=1$ as a coset representative of $(F^{\times})^2 \cdot1$.
Thanks to  \cite[Corollary to Proposition 1.7]{Cogdell-PS} along with Theorem \ref{connection}, we choose $W_i$ and $W_{b_j}$ mapping to characters $\pi^{(1)}_i$ and $\chi_{b_j}$ via the natural normalized projection map such that
 \[
 \begin{split}
   &I(s;\pi^{(1)}_i\times \pi^{(1)}_i,\varphi)
   =\int_{F^{\times}} \pi^{(1)}_i(z^2) \chi_{b_j}(z^2) \varphi(z^2) |z|^s\; d^{\times}z\\
   &=\int_{(F^{\times})^2} W_i \begin{pmatrix} z & \\ & 1 \end{pmatrix} W_{b_j} \begin{pmatrix} z & \\ & 1 \end{pmatrix} \varphi(z) |z|^{\frac{s}{2}-\frac{3}{4}}\; d^{\times} z
   =\int_{F^{\times}} W_i \begin{pmatrix} a & \\ & 1 \end{pmatrix} W_{b_j} \begin{pmatrix} a & \\ & 1 \end{pmatrix} \varphi(a) |a|^{\frac{s}{2}-\frac{3}{4}} \;d^{\times} a.\\
\end{split}
 \]
 With help of \cite[Lemma 2.5]{JO20-2}, this integral $I(s;\pi^{(1)}_i\times \pi^{(1)}_i,\varphi)$ equals to $I_{reg}(s;W_i^{\prime},W_{b_j})$ for some $W_i^{\prime} \in \mathcal{W}(\pi,\psi)$ and hence $s=s_0$ is a pole of order $d_{s_0}$ of $L_{reg}(s,\pi,\mathrm{Sym}^2)$.
 \end{proof}

We know from \citelist{\cite{JO20-2}*{\S 4}\cite{Cogdell-PS}} that $L(s,\chi,\wedge^2)=1$ and $L(s,\chi \times \chi)=L_{ex}(s,\chi \times \chi)$ for any character $\chi$. The main reference for the unexplained notation $L_{ex}(s,\chi \times \chi)$ is \cite[\S 2.2]{Cogdell-PS}. In order to obtain visibly consistent expressions of \cite[Theorem 2.1]{Cogdell-PS} or \cite{JO20-2}, it is preferable to write either $L_{ex}(s,\chi,\mathrm{Sym}^2)$ or $L(s,\chi,\mathrm{Sym}^2)$ for $L(s,\chi \times \chi)$ when it is clear from the context. Combining Proposition \ref{regular-factorization} with \eqref{product} yields the following theorem.

\begin{theorem}
\label{main-factorization}
Let $\pi$ be an irreducible admissible generic representation of $GL_2$ such that all of its derivatives are completely reducible and $\pi^{(0)}=\pi$. Then
\[
  L(s,\pi,\mathrm{Sym}^2)^{-1}=\underset{i,j}{l.c.m.}\{ L(s, \pi_i^{(j)}, \mathrm{Sym}^2 )^{-1} \}
\]
where the least common multiple is with respect to divisibility in $\mathbb{C}[q^{s/2},q^{-s/2}]$ and is taken over all $j$ with $0 \leq j \leq 1$ and for all constituents $\pi_i^{(1)}$ of $\pi^{(1)}$.
\end{theorem}

\section{Deformation and Specialization}
\label{sec4}

\subsection{Principal series representations and $L$-functions}
\label{sec4:1}

In \S \ref{sec4:1} we employ a deformation method proposed by Cogdell and Piatetski-Shapiro \cite[\S 3]{Cogdell-PS}. The advantage of this approach is to obtain a certain class of induced representations for which we can explicitly determine the poles of $L$-functions by means of shifting the location of those poles. For an admissible representation $\tau$ (which is not necessarily irreducible), we say that $\tau$ is of \textit{Whittaker type} if $\mathrm{Hom}_{GL_2}(\tau,\mathrm{Ind}_{N}^{GL_2}(\psi))$ is of dimension $1$ \cite{Cogdell-PS,JPSS,Matringe}. Let $\pi=\mathrm{Ind}^{GL_2}_B(\chi_1 \boxtimes \chi_2)$ be a normalized induced representation, where each $\chi_i$  is a character of $F^{\times}$. Let $\mathcal{D}_{\pi}$ denote the complex manifold $( \mathbb{C} \slash \frac{2 \pi i}{\log(q)} \mathbb{Z} )^2$. We denote by $\mathcal{D}_s$ the rescaled complex manifold $( \mathbb{C} \slash \frac{4 \pi i}{\log(q)} \mathbb{Z} )$. The isomorphism $\mathcal{D}_{\pi} \rightarrow (\mathbb{C}^{\times})^2$ is defined by $u=(u_1,u_2) \mapsto q^u:=(q^{u_1},q^{u_2})$. For each $u \in \mathcal{D}_{\pi}$ we set $\pi_u=\mathrm{Ind}^{GL_2}_B(\chi_1\nu^{u_1} \boxtimes \chi_2\nu^{u_2})$. Then $\pi_u$ is nothing but a representation of Whittaker type for every $u \in \mathcal{D}_{\pi}$.

\begin{definition}
\label{General}
We say that $u=(u_1,u_2) \in \mathcal{D}_{\pi}$ is in \textit{general position} if it satisfies the following conditions:
\begin{enumerate}[label=$(\mathrm{\arabic*})$] 
\item\label{General-1} A representation $\pi_u=\mathrm{Ind}^{GL_2}_B(\chi_1\nu^{u_1} \boxtimes \chi_2\nu^{u_2})$ is irreducible.
\item\label{General-2} The two characters $ \pi_u^{(1,0)}=\chi_2\nu^{u_2}$ and $\pi_u^{(0,1)}=\chi_1\nu^{u_1}$ are distinct.
\item\label{General-3} $L(s,\chi_1\nu^{u_1}, \mathrm{Sym}^2)$ and  $L(s,\chi_2\nu^{u_2}, \mathrm{Sym}^2)$ do not have any common poles.
\item\label{General-4} If $i \in \{ 1,2 \}$, then $L(s,\chi_i\nu^{u_i}, \mathrm{Sym}^2)$ and $L(s,\chi_1\nu^{u_1}\times \chi_2\nu^{u_2})$ do not have any common poles.
\item\label{General-5} If $s=e$ is a pole of $L(2s,\omega_{\pi}^2)$, then
the dimension of the space 
\[
\mathrm{Hom}_{GL_2}(\sigma\otimes\theta^{\psi} \otimes I(1,\omega^{-1}_{\sigma}), \mathbb{C} ) \;\; \text{with} \;\; \sigma=\mathrm{Ind}^{GL_2}_B(\chi_1\nu^{(u_1-u_2+e)/2} \boxtimes \chi_2\nu^{(-u_1+u_2+e)/2})
\]
is at most $1$.
\end{enumerate}
\end{definition}

The condition \ref{General-1} and \ref{General-2} assert that outside the hyperplane defining general position, all of the derivatives are completely reducible such that each constitute is irreducible and generic. As a consequence Theorem \ref{main-factorization} is applicable to the deformed representation $\pi_u$ in general position. The purpose of \ref{General-5} is to make it feasible to exploit the following result about distinguished representations which play an important role in our study of exceptional poles.

\begin{theorem}[\cite{Kaplan17}, Kaplan] \label{Kaplan} Let $\pi$ be an irreducible admissible generic representation of $GL_2$ and $\chi$ a character of $F^{\times}$.
\begin{enumerate}[label=$(\roman*)$]
\item $[\mathrm{Theorem\; 4.4} ]$ The induced representation $\mathrm{Ind}_B^{GL_2}(\chi \boxtimes \chi^{-1})$ is $\theta$-distinguished.
\item $[\mathrm{Corollary\; 4.19} ]$  If $\mathrm{Hom}_{GL_2}(\pi \otimes \theta^{\psi} \otimes \theta^{\psi^{-1}},\mathbb{C})\neq 0$, then $\pi \simeq \widetilde{\pi}$.
\end{enumerate}

\end{theorem}

We now confirm that off a finite number of the hyperplanes in $u$, the deformed representation $\pi_u$ is in general position.

\begin{proposition}
\label{hyperplane}
Let $\pi$ be as above, the element $u$ in $\mathcal{D}_{\pi}$ that is not in general position belongs to a finite number of affine hyperplanes. 
\end{proposition}

\begin{proof}
The conditions \ref{General-1}, \ref{General-2}, \ref{General-3}, and \ref{General-4} are explained in \cite[Proposition 5.1]{Matringe}. As described in \cite[Proposition 5.1]{Matringe}, \ref{General-5} can be checked along the line of \cite[Theorem 2.14]{Ya17}.
\end{proof}

The removed affine hyperplanes defining general position do not depend on $s \in \mathbb{C}$. Before going into the computation, let us recall what is known as Hartogs theorem in the view of the (complex) algebraic geometry.

\begin{theorem}[\citelist{\cite{Lojasiewicz}*{Chapter IV, \S 4, Theorem 4.6.7}}, Hartogs theorem\label{Hartogs}]
Let $\mathcal{M}$ be an $n$-dimensional complex manifold with $n \geq 2$. If $\mathcal{N} \subseteq \mathcal{M}$ is an analytic subset of codimension 2 or more, then every holomorphic function on $\mathcal{M}-\mathcal{N}$ extends to a holomorphic function on $\mathcal{M}$.
\end{theorem}

The following proposition is the starting point of computing $L$-functions for principal series representations and is needed for demonstrating the agreement of arithmetic and analytic $L$-functions which we shall proceed to provide.

\begin{proposition}
\label{converse}
Let $\pi=\mathrm{Ind}^{GL_2}_B(\chi_1 \boxtimes \chi_2)$ be a principal series representation of $GL_2$. Let $u=(u_1,u_2) \in \mathcal{D}_{\pi}$ be in general position and $\pi_u=\mathrm{Ind}^{GL_2}_B(\chi_1\nu^{u_1} \boxtimes \chi_2\nu^{u_2})$ the deformed representation. Then we have
\[
  L_{ex}(s,\mathrm{Ind}^{GL_2}_B(\chi_1\nu^{u_1} \boxtimes \chi_2\nu^{u_2}),\mathrm{Sym^2})=L_{ex}(s,\chi_1\nu^{u_1}  \times \chi_2\nu^{u_2}  ).
\]
\end{proposition}

\begin{proof}
First we suppose that $L_{ex}(s,\mathrm{Ind}^{GL_2}_B(\chi_1\nu^{u_1} \boxtimes \chi_2\nu^{u_2}),\mathrm{Sym^2})$ has a pole at $s=s_0$. As explained in the proof of Corollary \ref{general-characterization}, this is amount to saying that $L_{ex}(s,\nu^{\frac{s_0}{2}}\mathrm{Ind}^{GL_2}_B(\chi_1\nu^{u_1} \boxtimes \chi_2\nu^{u_2}),\mathrm{Sym^2})$ has a pole at $s=0$. Proposition \ref{Exceptional pole} says that the pole $s=0$ is simple and that it appears amongst that of $L(2s,\omega_{\pi_u\nu^{s_0/2}}^2)$. Unraveling the unramified twist, $L(2s,\omega^2_{\pi})$ has a poles at $s=e$ with $e=u_1+u_2+s_0$ and $\sigma$ in Definition \ref{General}-\ref{General-5} becomes $\pi_u\nu^{s_0/2}$. We apply Theorem \ref{main-characterization} to $\pi_u\nu^{s_0/2}$ and deduce from Theorem \ref{Kaplan} that $\pi_u\nu^{s_0/2}$ is self-contragredient. In other words
\[
 \mathrm{Ind}^{GL_2}_B(\chi_1\nu^{u_1} \boxtimes \chi_2\nu^{u_2})^{\sim} \simeq  \mathrm{Ind}^{GL_2}_B(\chi_1\nu^{u_1} \boxtimes \chi_2\nu^{u_2}) \nu^{s_0}.
\] 
The only way this is possible is that 
\begin{enumerate}[label=$(\roman*)$]
\item\label{relation-1}  $(\chi_1\nu^{u_1})^{\sim} \simeq \chi_1\nu^{u_1+s_0}$ \; \;and\;\; $(\chi_2\nu^{u_2})^{\sim} \simeq \chi_2\nu^{u_2+s_0}$
\item\label{relation-2}   $(\chi_1\nu^{u_1})^{\sim} \simeq \chi_2\nu^{u_2+s_0}$.
\end{enumerate}
In \ref{relation-1} the locus is defined by two independent equations $q^{-(2u_1+s)}\chi_1^2(\varpi)=1$ and $q^{-(2u_2+s)}\chi_2^2(\varpi)=1$, after evaluating both sides at $\varpi$, and hence will be of codimension $2$. Due to Bernstein's Theorem \cite[\S 3]{Cogdell-PS}, viewing $I(W_u,W_{\theta^{\psi}},f_{2s-1})$ as rational functions in $\mathbb{C}(q^{-s/2},q^{-u})$, every singularities of the integral $I(W_u,W_{\theta^{\psi}},f_{2s-1})$ must be accounted for the form \ref{relation-2} $(\chi_1\nu^{u_1})^{\sim} \simeq \chi_2\nu^{u_2+s_0}$ by Hartogs theorem, Theorem \ref{Hartogs}. As a result $L_{ex}(s,\chi_1\nu^{u_1} \times \chi_2\nu^{u_2})$ has a pole at $s=s_0$.

\par
Next we assume that $L_{ex}(s,\chi_1\nu^{u_1}  \times \chi_2\nu^{u_2})$ has a pole at $s=s_0$.  Equivalently $L_{ex}(s,\chi_1\nu^{u_1+\frac{s_0}{2}}  \times \chi_2\nu^{u_2+\frac{s_0}{2}})$ possesses a pole at $s=0$. Then we have $(\chi_1\nu^{u_1+\frac{s_0}{2}})^{\sim} \simeq \chi_2\nu^{u_2+\frac{s_0}{2}}$ and the principal series representation $\pi_u\nu^{\frac{s_0}{2}}=\mathrm{Ind}^{GL_2}_B(\chi_1\nu^{u_1+\frac{s_0}{2}} \boxtimes \chi_2\nu^{u_2+\frac{s_0}{2}})$ is self-contragredient. In terms of central characters, this says that $\omega^2_{\pi_u\nu^{s_0/2}}$ is trivial. Now Theorem \ref{Kaplan} assures that $\pi_u\nu^{s_0/2}$ is $\theta$-distinguished. At this point, we essentially repeat the argument of \cite[Theorem 3.7-(1)]{Ya17} for completeness. We further assume that $L_{ex}(s,\pi_u\nu^{\frac{s_0}{2}},\mathrm{Sym}^2)$ is holomorphic at $s=0$. Since $L$-functions $L(s,\chi_1\nu^{u_1+\frac{s_0}{2}}  \times \chi_2\nu^{u_2+\frac{s_0}{2}})^{-1}$, $L(s,\chi_1\nu^{u_1+\frac{s_0}{2}},\mathrm{Sym}^2 )^{-1}$  and $L(s,\chi_2\nu^{u_2+\frac{s_0}{2}},\mathrm{Sym}^2 )^{-1}$ are relatively prime, Theorem \ref{main-factorization} implies that $L(s,\pi_u\nu^{\frac{s_0}{2}},\mathrm{Sym}^2)$ is holomorphic at $s=0$. Remembering a quadratic character $\omega_{\pi_u\nu^{s_0/2}}$, Lemma 3.16 in \cite{Ya17} enables us to take $W \in \mathcal{W}(\pi_u\nu^{\frac{s_0}{2}},\psi)$, $W_{\theta^{\psi}} \in \mathcal{W}(\theta^{\psi},\psi^{-1})$ and $f_{2s-1} \in V_{good}(2s-1,\omega_{\pi_u\nu^{s_0/2}}^{-1})$ such that
\begin{equation}
\label{choices}
  M(1,\omega_{\pi_u\nu^{s_0/2}}^{-1})f_1=0 \quad \text{and} \quad \lim_{s \rightarrow 1} I(W,W_{\theta^{\psi}},\hat{N}(2s-1,\omega^{-1}_{\pi_u\nu^{s_0/2}})f_{2s-1}) \neq 0.
\end{equation}
Evaluating the both side of the functional equation \eqref{func-Nhat} at $s=1$, we arrive at
\[
  0 \neq \frac{I(W,W_{\theta^{\psi}},\hat{N}(1,\omega^{-1}_{\pi_u\nu^{s_0/2}})f_{1})}{L(0,\pi_u\nu^{s_0/2},\mathrm{Sym}^2)}
  =\varepsilon(1,\pi,\mathrm{Sym}^2,\psi) \frac{I(W,W_{\theta^{\psi}},f_{1})}{L(1,\pi_u\nu^{s_0/2},\mathrm{Sym}^2)}.
\]
But the first condition of \eqref{choices} enforces that the linear functional
\[
 (W,W_{\theta^{\psi}},f_1) \mapsto \varepsilon(1,\pi,\mathrm{Sym}^2,\psi) \frac{I(W,W_{\theta^{\psi}},f_{1})}{L(1,\pi_u\nu^{s_0/2},\mathrm{Sym}^2)}
\]
cannot factor through the quotient map $ \pi_u\nu^{s_0/2} \otimes\theta^{\psi} \otimes I(1,\omega^{-1}_{\pi_u\nu^{s_0/2}})  \rightarrow \pi_u\nu^{s_0/2} \otimes\theta^{\psi} \otimes \theta^{\psi^{-1}}\otimes \omega_{\pi_u\nu^{s_0/2}}$ and hence $\pi_u\nu^{s_0/2}$ is not $\theta$-distinguished making the use of Proposition \ref{char-Yamana}. This is a contradiction and we conclude that $L_{ex}(s,\pi_u\nu^{\frac{s_0}{2}},\mathrm{Sym}^2)$ has a pole at $s=0$, which is equivalent to saying that $L_{ex}(s,\pi_u,\mathrm{Sym}^2)$ has a pole at $s=s_0$.
\end{proof}

We shift our gear to the principal series representation $\pi=\mathrm{Ind}^{GL_2}_B(\chi_1 \boxtimes \chi_2)$ of $GL_2$. The representation is a possibly reducible representation. Nonetheless all the theories in Section \ref{sec3:2} still go through and all the constructions of $I(W,W_{\theta^{\psi}},f_{2s-1})$, $L(s,\pi,\mathrm{Sym}^2)$, $L_{ex}(s,\pi,\mathrm{Sym}^2)$ and $L_{reg}(s,\pi,\mathrm{Sym}^2)$ are completely carried over without any adjustments.

\begin{theorem}
\label{deform-multiplicativity}
Let $\pi=\mathrm{Ind}^{GL_2}_B(\chi_1 \boxtimes \chi_2)$ be a principal series representation of $GL_2$. Let $u=(u_1,u_2) \in \mathcal{D}_{\pi}$ be in general position and $\pi_u=\mathrm{Ind}^{GL_2}_B(\chi_1\nu^{u_1} \boxtimes \chi_2\nu^{u_2})$ the deformed representation. Then we have the following:
\begin{enumerate}[label=$(\roman*)$]
\item \label{deform-multiplicativity-1}
$\displaystyle
 L(s,\mathrm{Ind}^{GL_2}_B(\chi_1\nu^{u_1} \boxtimes \chi_2\nu^{u_2}),\mathrm{Sym^2})=L(s+u_1+u_2,\chi_1  \times \chi_2  )\prod_{1 \leq i \leq 2} L(s+2u_i,\chi_i,\mathrm{Sym}^2).
$
\item \label{deform-multiplicativity-2}
$L(s,\mathrm{Ind}^{GL_2}_B(\chi_1 \boxtimes \chi_2),\mathrm{Sym^2})^{-1}$ divides $\displaystyle L(s,\chi_1  \times \chi_2  )^{-1}\prod_{1 \leq i \leq 2} L(s,\chi_i,\mathrm{Sym}^2)^{-1}$ in $\mathbb{C}[q^{\pm s/2}]$, that is, there is a $Q(X) \in \mathbb{C}[X]$ such that
\[
 L(s,\mathrm{Ind}^{GL_2}_B(\chi_1 \boxtimes \chi_2),\mathrm{Sym^2})=Q(q^{-s/2}) L(s,\chi_1  \times \chi_2  )\prod_{1 \leq i \leq 2} L(s,\chi_i,\mathrm{Sym}^2).
\]
\end{enumerate}
\end{theorem}

\begin{proof}
The first item is a direct consequence of Proposition \ref{converse} together with Theorem \ref{main-factorization}. Regarding the second item, it follows from the standard Bernstein's argument of the continuation principle \cite[\S 3]{Cogdell-PS} that $I(W_u,W_{\theta^{\psi}},f_{2s-1})$ defines a rational function in $\mathbb{C}(q^{\pm s/2},q^{\pm u})$ as we vary in $u$. According to \ref{deform-multiplicativity-1}, the fraction
\begin{equation}
\label{deform-ratio}
  \frac{I(W_u,W_{\theta^{\psi}},f_{2s-1})}{L(s+u_1+u_2,\chi_1  \times \chi_2  )\prod_{1 \leq i \leq 2} L(s+2u_i,\chi_i,\mathrm{Sym}^2)}
\end{equation}
has no poles on Zariski open set of $u$ in general position. Our remaining task is to the ration \eqref{deform-ratio} in fact lies in $\mathbb{C}[q^{\pm s/2},q^{\pm u}]$.
But the proof of holomorphy continues as in \cite[Proposition 4.1]{Cogdell-PS}, (See \citelist{\cite{JO-3}*{Proposition 5.3} \cite{Matringe}*{Lemma 5.1}} for further refinements of basic ideas). If we now specialize at $u=0$ we conclude that
\[
  \frac{I(W,W_{\theta^{\psi}},f_{2s-1})}{L(s,\chi_1  \times \chi_2  )\prod_{1 \leq i \leq 2} L(s,\chi_i,\mathrm{Sym}^2)}
\]
has no poles for all $W \in \mathcal{W}(\pi,\psi)$, $W_{\theta^{\psi}} \in \mathcal{W}(\theta^{\psi},\psi^{-1})$ and $f_{2s-1} \in V_{good}(2s-1,\omega_{\pi}^{-1})$.
\end{proof}

As elucidated in the Langlands-Shahidi method \cite{CoShTs,Ganapathy-Lomeli,Shahidi}, the unit appearing in Proposition \ref{Weak-multiplicativity}-\ref{Weak-multiplicativity-2} will be presumably $1$. This is so-called the \textit{multiplicativity} of $\gamma$-factors. However manifesting the multiplicativity property requires manipulating integrals in a delicate manner. Apparently the best we can do at this point is to obtain the weak multiplicativity which is enough for the applications therein.

\begin{proposition}
\label{Weak-multiplicativity}
Let $\pi=\mathrm{Ind}^{GL_2}_B(\chi_1 \boxtimes \chi_2)$ be a principal series representation of $GL_2$. Let $u=(u_1,u_2) \in \mathcal{D}_{\pi}$ be in general position and $\pi_u=\mathrm{Ind}^{GL_2}_B(\chi_1\nu^{u_1} \boxtimes \chi_2\nu^{u_2})$ the deformed representation. Then we have the following:
\begin{enumerate}[label=$(\roman*)$]
\item\label{Weak-multiplicativity=1}  $\gamma(s,\pi_u,\mathrm{Sym}^2,\psi)$ and $\displaystyle \gamma(s+u_1+u_2,\chi_1 \times \chi_2,\psi) \prod_{1 \leq i \leq 2} \gamma(s+2u_i,\chi_i,\mathrm{Sym}^2,\psi)$
are equal up to a unit in $\mathbb{C}[q^{\pm s/2},q^{\pm u}]$.
\item\label{Weak-multiplicativity-2} $\gamma(s,\pi,\mathrm{Sym}^2,\psi)$ and $\displaystyle \gamma(s,\chi_1 \times \chi_2,\psi) \prod_{1 \leq i \leq 2} \gamma(s,\chi_i,\mathrm{Sym}^2,\psi)$ are equal up to a unit in $\mathbb{C}[q^{\pm s/2}]$.
\end{enumerate}
\end{proposition}

\begin{proof}
The proof is standard  \citelist{\cite{JO-3}*{Proposition 5.4} \cite{Matringe}*{Proposition 5.5}} and is due to Cogdell and Piatetski-Shapiro \cite[Proposition 4.3]{Cogdell-PS} by applying Proposition \ref{func-epsilon}, Theorem \ref{deform-multiplicativity}, and Proposition \ref{constant} to our setting.
\end{proof}

To proceed further we adopt the terminology from \cite{Cogdell-PS,Matringe}. We say that an admissible representation of $GL_2$ is of \textit{Langlands type} if $\Pi$ is of the form $\mathrm{Ind}^{GL_2}_{B}(\pi_1\nu^{u_1} \boxtimes \pi_2\nu^{u_2})$, where each $\pi_i$ is a square integrable representation of $GL_{n_i}$, $n_1+n_2=2$, each $u_i$ is real and they are ordered so that $u_1 \geq u_2$. Let $\pi$ be an irreducible admissible representation of $GL_2$. A representation of the form $\chi \circ \mathrm{det}$ with $\chi$ a character of $F^{\times}$ is an example of an irreducible and admissible but not generic representation of $GL_2$. Regardless of being generic, $\pi$ can be realized as the unique Langlands quotient of Langlands type $\Pi=\mathrm{Ind}^{GL_2}_{B}(\pi_1\nu^{u_1} \boxtimes \pi_2\nu^{u_2})$ which is of Whittaker type. The $L$-function $L(s,\pi,\mathrm{Sym}^2)$ is defined to be
\[
 L(s,\pi,\mathrm{Sym}^2):=L(s,\Pi,\mathrm{Sym}^2).
\]

\begin{lemma}
\label{JPSS-lemma}
Let $\pi=\mathrm{Ind}^{GL_2}_{B}(\pi_1 \boxtimes \pi_2)$ be a principal series representation of $GL_2$ with each $\pi_i$ a character of $F^{\times}$ (possibly a representation of Langlands type). Then $L(s,\pi_2,\mathrm{Sym}^2)^{-1}$ divides $L(s,\pi,\mathrm{Sym}^2)^{-1}$ in $\mathbb{C}[q^{\pm s/2}]$, that is, $L(s,\pi_2,\mathrm{Sym}^2)=Q(q^{-s/2})L(s,\pi,\mathrm{Sym}^2)$ for some $Q(X) \in \mathbb{C}[X]$.
\end{lemma}

\begin{proof}
Jacquet, Piatetski-Shapiro, and Shalika established in \cite[Proposition 9.1]{JPSS} that for a character $\pi_2$ and a Schwartz-Bruhat function $\varphi \in \mathcal{S}(F)$, there exists $W \in \mathcal{W}(\pi,\psi)$ such that
\begin{equation}
\label{Whittaker-induced}
  W \begin{pmatrix} a & \\ & 1 \end{pmatrix}=\pi_2(a)\varphi(a)|a|^{\frac{1}{2}}.
\end{equation}
The divisibility of $L$-functions follows exactly in the same way as we argued in the proof of Proposition \ref{regular-factorization}, appealing to the property above \eqref{Whittaker-induced} instead of  \cite[Corollary to Proposition 1.7]{Cogdell-PS}.
\end{proof}

We denote by $P \sim Q$ that the ratio is a unit in $\mathbb{C}[q^{s/2},q^{-s/2}]$ for two rational functions $P(q^{-s/2})$ and $Q(q^{-s/2})$ in $\mathbb{C}(q^{-s/2})$.
Unlike the Langlands-Shahidi method \cite{CoShTs,Ganapathy-Lomeli} which takes the following formalism \eqref{Langlands-prod} to be the definition, proving the inductive formula \eqref{Langlands-prod} is the crucial subject in the theory of Rankin-Selberg integrals.

\begin{theorem}
\label{Langlands}
Let $\pi=\mathrm{Ind}^{GL_2}_B(\pi_1\nu^{u_1} \boxtimes \pi_2\nu^{u_2})$ be a representation of Langlands type of $GL_2$. Then we have
\begin{equation}
\label{Langlands-prod}
 L(s,\pi,\mathrm{Sym}^2)=L(s+u_1+u_2,\pi_1  \times \pi_2  )\prod_{1 \leq i \leq 2} L(s+2u_i,\pi_i,\mathrm{Sym}^2).
\end{equation}
\end{theorem}

\begin{proof}
The proof is identical with those employed in \citelist{\cite{Cogdell-PS}*{Theorem 4.1}\cite{Matringe}*{Theorem 5.1}}.
\end{proof}

As a result, Theorem \ref{Langlands} in the case of irreducible unramified representations yields:

\begin{corollary}
\label{unramified}
Let $\pi=\mathrm{Ind}^{GL_2}_B(\mu_1 \boxtimes \mu_2)$ be an irreducible unramified representation of $GL_2$. Then we have
\[
 L(s,\pi,\mathrm{Sym}^2)=\prod_{1 \leq i \leq j \leq 2} \frac{1}{1-\mu_i(\varpi)\mu_j(\varpi)q^{-s}}.
\]
\end{corollary}

Corollary \ref{unramified} provides an affirmative answer to Yamana's question \cite[\S 3H]{Ya17} for the unramified $GL_2$-case at least.

\subsection{The Langlands correspondence and the equality of $L$-functions} We denote by $W^{\prime}_F$ the Weil-Deligne group of $F$ (see for example \cite{Tate}). 
We then denote by $L(s,\mathrm{Sym}^2(\rho))$ the Artin $L$-function attached to the $2$-dimensional (complex) Frobenius semi-simple representation $\rho$ of  $W^{\prime}_F$, where $\mathrm{Sym}^2(\rho)$ is the symmetric square of $\rho$.

\begin{theorem}
Let $\pi \mapsto \rho(\pi)$ be the local Langlands correspondence from the set of isomorphism classes of $2$-dimensional complex representations of $W_F^{\prime}$ to that of the isomorphism classes of irreducible admissible representations of $GL_2(F)$. Then we have the following equality of $L$-functions:
\[
  L(s,\pi,\mathrm{Sym}^2)=L(s,\mathrm{Sym^2(\rho(\pi))}).
\]
\end{theorem}

\begin{proof}
The case of $\pi$ a discrete series representation (also called a special or Steinberg representation) was established in \cite{Ya17}, so it remains to
consider the case when $\pi$ is the induced representation of Langlands type. The proof in this case is a consequence of Theorem \ref{Langlands} accompanied with the local Langlands correspondence. We refer the reader to the beginning of \cite[Section  5.3]{Matringe} for further details.
\end{proof}

We turn attention to the factorization
\[
 L(s,\rho(\pi) \otimes \rho(\pi))=L(s,\mathrm{\wedge^2(\rho(\pi))})L(s,\mathrm{Sym^2(\rho(\pi))})
\]
according to Langlands formalism. We derive from \cite[Theorem 5.14]{JO-3} and \cite[Proposition 4.1]{JO20-2} pertaining $L(s,\pi,\wedge^2)=L(s,\omega_{\pi})$  that the analogous factorization is available in the realm of integral representations. 

\begin{corollary}
Let $\pi$ be an irreducible admissible representation of $GL_2(F)$. Then we obtain
\[
 L(s,\pi \times \pi)=L(s,\pi,\wedge^2)L(s,\pi,\mathrm{Sym}^2)=L(s,\omega_{\pi})L(s,\pi,\mathrm{Sym}^2),
\]
where $L(s,\pi,\wedge^2)$ is the Jacquet-Shalika $L$-function defined in \cite{JO20-2}.
\end{corollary}

\section{Stability of $\gamma$-factors}
\label{sec5}

\subsection{Howe vectors and Bessel functions}
\label{sec5:1}
In \S \ref{sec5:1}, we extend the basic theory of Howe vectors \cite{ChaiZhang,Zhang} to metaplectic groups $\widetilde{GL}_2$. Each of the constructions in turn had its origin in the work of R. Howe \cite{Howe}. We assume that $\psi$ is unramified, that is, $\psi$ is trivial on $\mathcal{O}$ while $\psi(\varpi^{-1}\mathcal{O}) \neq 1$. We put $d_m=t(\varpi^{-m},\varpi^m)$
and let $J_m=d_mK_md_m^{-1}$. Then $J_m$ is given by
\[
  J_m=\begin{pmatrix}  1+\mathfrak{p}^m & \mathfrak{p}^{-m} \\ \mathfrak{p}^{3m} &  1+\mathfrak{p}^m  
  \end{pmatrix}.
\]
For $k=(k_{ij}) \in K_m$, we define a character $\tau_m$ of $K_m$ by
\[
 \tau_m(k)=\psi(\varpi^{-2m}k_{12}).
\]
As $\psi$ is unramified, we can check that $\tau_m$ is indeed a character on $K_m$. We define a character $\psi_m$ on $J_m$ by
\[
  \psi_m(j)=\tau_m(d_m^{-1}jd_m) \;\; \text{for} \;\; j \in J_m.
\]
We can also see that $\psi$ and $\psi_m$ agree on $N_m$ where $N_m=N \cap J_m$. For the rest of Section \ref{sec5} we let $\pi$ be an irreducible admissible representation of $GL_2$. If $\pi$ is generic $\pi \simeq \mathcal{W}(\pi,\psi)$ while if $\pi$ is not generic then we set $\mathcal{W}(\Pi,\psi)$ to be a Whittaker model for $\pi$ where $\Pi$ is the induced representation of Langlands type having $\pi$ as its unique irreducible quotient. We fix a Whittaker function $W \in \mathcal{W}(\pi,\psi)$ such that $W(I_2)=1$. For $m \geq 1$,
we define a function $W_m$ on $GL_2$ by
\[
  W_m(g)=\frac{1}{\mathrm{vol}(N_m)} \int_{N_m} W(gn) \psi^{-1}(n) \; dn.
\]
 The construction of Howe Whittaker functions can be adapted to the metaplectic setting with $GL_2$ replaced by $\widetilde{GL}_2$. We take $W_{\theta^{\psi}} \in \mathcal{W}(\theta^{\psi},\psi^{-1})$ to be $W_{\theta^{\psi}}(I_2)=1$. In like manner we define a function $W_{\theta^{\psi},m}$ on $\widetilde{GL}_2$ for $m \geq 1$ by
\[
  W_{\theta^{\psi},m}(\widetilde{g})=\frac{1}{\mathrm{vol}(N_m)} \int_{N_m} W_{\theta^{\psi}}(\widetilde{g}\mathfrak{s}(n)) \psi(n) \; dn.
\]
$ W_{\theta^{\psi},m}$ is simply $W_1$ in the notation of \cite[Lemma 4.1.1]{Gelbart-PS}. Many of results in \S \ref{sec5:1} resemble those in the setting of the general linear group $GL_2$ \citelist{\cite{Soudry}*{Lemma 4.1}\cite{Zhang}*{\S 5.1}}, henceforth we only remark on the nature of differences or omit the proof most of the time.

\begin{lemma}
\label{Howe}
We choose $W_{\theta^{\psi}} \in \mathcal{W}(\theta^{\psi},\psi^{-1})$ so that $W_{\theta^{\psi}}(I_2)=1$. Let $L$ be a positive integer such that $\theta^{\psi}(\mathfrak{s}(K_L))W_{\theta^{\psi}}=W_{\theta^{\psi}}$. Then we have
\begin{enumerate}[label=$(\mathrm{\arabic*})$] 
\item\label{Howe-1} $W_{\theta^{\psi},m}(I_2)=1$.
\item\label{Howe-2} If $m \geq L$ then $W_{\theta^{\psi},m}(\widetilde{g}\mathfrak{s}(j))=\psi_m^{-1}(j)W_{\theta^{\psi},m}(\widetilde{g})$ for all $j \in J_m$.
\item\label{Howe-3} If $m \ge k$, then
\[
  W_{\theta^{\psi}.m}(\widetilde{g})=\frac{1}{\mathrm{vol}(N_m)} \int_{N_m} W_{\theta^{\psi},k}(\widetilde{g}\mathfrak{s}(n)) \psi(n) dn.
\] 
\end{enumerate}
\end{lemma}

We now fix $W \in \mathcal{W}(\pi,\psi)$ and $W_{\theta^{\psi}} \in \mathcal{W}(\theta^{\psi},\psi^{-1})$ satisfying $W(I_2)=W_{\theta^{\psi}}(I_2)=1$. Let $L$ be an integer such that $W$ and $W_{\theta^{\psi}}$ are right invariant under $K_L$. The vector $W_{\theta^{\psi},m}$ is called \textit{Howe vectors} of $\theta^{\psi}$ if $m \geq L$. A parallel notation $W_m$ is applied to $W$.

\begin{lemma}
\label{excep-a-support}
Let $m \geq L$ and $W_{\theta^{\psi},m}$ be as in Lemma \ref{Howe}. Then we have
\begin{enumerate}[label=$(\mathrm{\arabic*})$] 
\item\label{excep-a-support-1} $W_{\theta^{\psi},m}(\mathfrak{s}(t(a,1))) \neq 0$ if and only if $a \in 1+\mathfrak{p}^m$.
\item\label{excep-a-support-2} If $W_{\theta^{\psi},m}(\mathfrak{s}(t(a,1))\mathfrak{s}(w_2)) \neq 0$, then $a \in \mathfrak{p}^{-3m}$.
\end{enumerate}
\end{lemma}

In virtue of Lemma \ref{Howe}-\ref{Howe-2} the Whittaker function $W_{\theta^{\psi},m}$ satisfies 
\[
  W_{\theta^{\psi},m}(\mathfrak{s}(n_1)\widetilde{g}\mathfrak{s}(n_2))=\psi^{-1}(n_1)\psi^{-1}(n_2)W_{\theta^{\psi},m}(\widetilde{g}) \;\; \text{for all} \;\; n_1 \in N, n_2 \in N_m \;\; \text{and} \;\; \widetilde{g} \in \widetilde{GL}_2.
\]
The transformation implies that $W_{\theta^{\psi},m}$ really behaves partially like a Bessel function and $W_{\theta^{\psi},m}$ is said to be a \textit{partial Bessel function}. In particular we are interested in the full Bessel function attached to the representation $\theta^{\psi}$ for the further analysis. For $W_{\theta^{\psi}} \in \mathcal{W}(\theta^{\psi},\psi^{-1})$ the integral
\[
  \ell(W,x)=\int_{N_m} W_{\theta^{\psi}} \left( \mathfrak{s} \begin{pmatrix} x & \\ & 1 \end{pmatrix} \mathfrak{s}(w_2) \mathfrak{s}(n) \right) \psi(n) dn
\]
converges in the sense of stabilizing for large $m$ depending on $x$. See \cite[Lemma 4.1]{Gelbart-PS}. It defines a Whittaker functional on $V_{\theta^{\psi}}$ for fixed $x \in F$, as does the functional $v \mapsto  W_{\theta^{\psi},v}(I_2)$. It follows from the uniqueness of Whittaker functional \cite[\S 2.1]{Gelbart-PS} that there is the constant proportionality as a function of $x$ when $v$ varies, that is,
\[
  j_{\theta^{\psi}}(x)W_{\theta^{\psi},v}(I_2)=\int_{N_m} W_{\theta^{\psi},v} \left( \mathfrak{s} \begin{pmatrix} x & \\ & 1 \end{pmatrix} \mathfrak{s}(w_2) \mathfrak{s}(n) \right) \psi(n) dn.
\]
 The function $ j_{\theta^{\psi}}(x)$ is called the \textit{Bessel function}. The basic behavior of Bessel functions is carried out as a part of Ph.D. Thesis by Wang \cite[p. 37]{Wang} supervised by Rallis. The proof is in turn based on the work of Gelbart and Piatetski-Shapiro \cite[Proposition 4.4.2]{Gelbart-PS} and Soudry \cite[Lemma 4.1]{Soudry}.

\begin{proposition}
\label{Bessel}
Let $L \geq 1$ and $W_{\theta^{\psi}} \in \mathcal{W}(\theta^{\psi},\psi^{-1})$ be as in Lemma \ref{Howe}.
\begin{enumerate}[label=$(\mathrm{\arabic*})$] 
\item\label{Bessel-1} (Asymptotics for Bessel function on a certain range) For every $m \geq L$, there is a positive constant $C_m=q^{9m}$ such that if $|x| > C_m$, then
\[
  j_{\theta^{\psi}}(x)=\int_{xu^{-2} \in 1+\mathfrak{p}^{3m}} (-x,u)_F \mu_{\psi}(u)^{-1}\psi \left( xu^{-1}+u\right) du.
\]
\item\label{Bessel-2} (Bessel functions) Let $M$ be a positive integer such that $q^{6M} \geq C_L$. If $i \geq M$, then
\[
  \mathrm{vol}(N_{3i}) W_{\theta^{\psi},3i}(\mathfrak{s}(t(x,1))\mathfrak{s}(w_2))=j_{\theta^{\psi}}(x) \quad \text{for all} \quad |x| \leq q^{6i}.
\]
\item\label{Bessel-3} (Local uniform smoothness) Let $i \geq M$. For $x \in \mathfrak{p}^{-6i}$ and $a \in 1+\mathfrak{p}^{3i}$,
\[
  j_{\theta^{\psi}}(xa)=j_{\theta^{\psi}}(x).
\]
 \end{enumerate}
\end{proposition}


\subsection{Intertwining operators}
\label{sec5:2}

In \S \ref{sec5:2}, we are concerned with computing the image of the test function \eqref{char} under the intertwining operator. We denote by $\chi$ a character of $F^{\times}$. Let $X$ be an open compact subgroup of $N$. For $x \in X$ and $i > 0$, we set 
\[
  A(x,i)=\{ \overline{n} \in \overline{N} \; | \; \overline{n}x \in Z^2P\overline{N}_i\}.
\]

\begin{lemma}
\label{compact-decomp}
For any positive integer $c$ there exists an integer $i_1=i_1(X,c)$ such that for all $i \geq i_1$, $x \in X$ and $\overline{n} \in A(x,i)$ we have
\begin{equation}
\label{eq-decomposition}
  \overline{n}x=t(\alpha,\alpha)t(\beta,1)n\overline{n}_0
\end{equation}
with $\alpha, \beta \in 1+\mathfrak{p}^c$, $n \in N$ and $\overline{n}_0 \in \overline{N}_i$. Furthermore there is an integer $i_2=i_2(X)$ such that for all $i \geq i_2$
we have $A(x,i)=\overline{N}_i$.
\end{lemma}

\begin{proof}
We mimic a train of the standard argument from \cite[Lemma 3.7]{ChaiZhang}. Since $\overline{n}(y)n(x) \in Z^2P\overline{N}_i$, $\overline{n}(y)n(x)$ takes a form
\[
 \overline{n}(y)n(x) =\begin{pmatrix} \alpha\beta & \upsilon \\ & \alpha \end{pmatrix}\overline{n}(\overline{y})
\]
for $\alpha \in (F^{\times})^2$, $|x| < M$ and $\overline{y} \in \mathfrak{p}^{3i}$.
For convenience, we rewrite it as
\begin{equation}
\label{matrix-mul}
 \begin{pmatrix} 1-x\overline{y} & x \\ -\overline{y} & 1 \end{pmatrix}=n(x)\overline{n}(-\overline{y})=\overline{n}(-y) \begin{pmatrix} \alpha\beta & \upsilon \\ & \alpha \end{pmatrix}
 =\begin{pmatrix} \alpha\beta & \upsilon \\ -\alpha\beta y & -\upsilon y+\alpha \end{pmatrix}.
\end{equation}
Looking at the left entries in \eqref{matrix-mul}, we have $1-x\overline{y}=\alpha\beta$ and $\overline{y}=\alpha\beta y$. It is quickly seen from the condition $|x| < M$ that for any positive integer $c$ we can pick $i_1(X,c)$ such that $\alpha\beta=1-x\overline{y} \in 1+\mathfrak{p}^c$. Taking the determinant on the both sides of \eqref{matrix-mul}, we obtain $\alpha^2\beta=1$, which tells us that $\alpha \in 1+\mathfrak{p}^c$. Hence $\beta$ must be in $1+\mathfrak{p}^c$. Now we put $i_2(X):=i_1(X,1)$ and get $y \in \mathfrak{p}^{3i}$ from $y=(\alpha\beta)^{-1}\overline{y}$.
\end{proof}

Let us rewrite \eqref{char} in the current setting for later use. Given a positive integer $i$ and a complex number $s \in \mathbb{C}$, we define the following function $f^i_{2s-1} \in I(2s-1,\omega_{\pi \otimes \chi}^{-1})$ on $\widetilde{GL}_2$ by
\[
f_{2s-1}^i(\widetilde{g})=
 \begin{cases}
  \xi\omega^{-1}_{\pi \otimes \chi}(a)  \delta_B(p)^{(2s+1)/4} & \text{if}\;\;  \widetilde{g}=(1,\xi)\mathfrak{s}(t(a,a)) \mathfrak{s}(p) \mathfrak{s}(\overline{n})\;\; \text{with}\;\; a \in (F^{\times})^2, p \in P, \overline{n} \in \overline{N}_i, \\
  0 &  \text{otherwise}.
 \end{cases}\\
\]
The section is the $\widetilde{GL}_2$-analogue to that of \cite[Lemma 3.8]{ChaiZhang}. We let $\widetilde{f}^i_{1-2s}$ denote $M(2s-1,\omega^{-1}_{\pi \otimes \chi})f^i_{2s-1}$ the image of $f^i_{2s-1}$ under the intertwining operator $M(2s-1,\omega^{-1}_{\pi \otimes \chi})$. We now evaluate $f^i_{2s-1}$.

\begin{proposition}
\label{eval-intertwining}
Let $X$ be a compact open subset of $N$. Then there exists an integer $I(X,\omega_{\pi\otimes\chi})$ such that for all $i \geq I(X,\omega_{\pi\otimes\chi})$ and $x \in X$, we have
\[
  \widetilde{f}^i_{1-2s}(\mathfrak{s}(w_2)\mathfrak{s}(n(x)))=\mathrm{vol}(\overline{N}_i).
  \]
\end{proposition}

\begin{proof}
We choose a positive integer $c$ such that  $\omega_{\pi \otimes \chi}$ is trivial on $1+\mathfrak{p}^c$. Then we choose $I(X,\omega_{\pi\otimes\chi})$ to be the maximum of $i_1(X,c)$ and $i_2(X)$. Since
\[
 \widetilde{f}^i_{1-2s}(\mathfrak{s}(w_2)\mathfrak{s}(n(x)))=\int_F f_{2s-1}(\mathfrak{s}(w_2)\mathfrak{s}(n(y)\mathfrak{s}(w_2)\mathfrak{s}(n(x))) dx,
\]
the value of $\widetilde{f}^i_{1-2s}$ is completely determined by the support of the function $f_{2s-1}(\mathfrak{s}(\overline{n}(y)n(x)))$. In fact $\overline{N}^{\ast}N^{\ast}=\mathfrak{s}(\overline{N}N)$ because $\sigma_2(\overline{n}(y),n(x))=1$ for all $x, y \in F$. Then elements $\overline{n}(y)n(x)$ in $Z^2P\overline{N}$ contribute to the support. By writing $\overline{n}(y)n(x)=t(\alpha,\alpha)t(\beta,1)n\overline{n}_0$ as in \eqref{eq-decomposition}, Lemma \ref{compact-decomp} implies that
\[
f_{2s-1}(\mathfrak{s}(w_2)\mathfrak{s}(n(y)\mathfrak{s}(w_2)\mathfrak{s}(n(x)))=\omega^{-1}_{\pi\otimes\chi}(\alpha)\chi^{-2}(\alpha)\delta^{\frac{2s+1}{4}}_B(t(\beta,1))f(\mathfrak{s}(\overline{n}_0))= 1.
\]
But again by Lemma \ref{compact-decomp}, this only happens when $n(x) \in A(x,i)=\overline{N}_i$. Therefore we have $ \widetilde{f}^i_{1-2s}(\mathfrak{s}(w_2)\mathfrak{s}(n(x)))=\mathrm{vol}(\overline{N}_i)$ as claimed.
\end{proof}

\subsection{Dependence on $\psi$}
\label{sec5:3}

In order to move to twisted Rankin-Selberg integrals, we need some preparation. For $W \in \mathcal{W}(\pi,\psi)$, $W_{\theta^{\psi}} \in \mathcal{W}(\theta^{\psi},\psi^{-1})$, $f_{2s-1} \in I(2s-1,\omega^{-1}_{\pi\otimes\chi})$ and a complex number $s$, we define the integral
\[
  I(W,W_{\theta^{\psi}},f_{2s-1},\chi)=\int_{Z^2N \backslash GL_2} W(g)W_{\theta^{\psi}}(\mathfrak{s}(g)) f_{2s-1}(\mathfrak{s}(g)) \chi(\mathrm{det}(g)) dg.
\]
By the usage of \eqref{epsilon}, we reframe the functional equation of the $\gamma$-factor \eqref{func-Nhat} by
\[
I(W,W_{\theta^{\psi}},\hat{N}(2s-1,\omega^{-1}_{\pi})f_{2s-1},\chi) =\gamma(s,\pi \otimes \chi,\mathrm{Sym}^2,\psi)I(W,W_{\theta^{\psi}},f_{2s-1},\chi).
\]
The main theme of \S \ref{sec5:3} is to form so-called the dependence on $\psi$, analogous to \cite{Gelbart-PS}*{Lemma 5.1.3} and \cite[Lemma 2.8]{Zhang}, in the setting of $\gamma(s,\pi \otimes \chi,\mathrm{Sym}^2,\psi)$.

\begin{proposition}
For any $a \in F^{\times}$, let $\mu_{\psi_a}(a)=\mu_{\psi}(a)(a,a)_F$. Then
\[
  \gamma(s,\pi \otimes \chi,\mathrm{Sym}^2,\psi_a)=\mu_{\psi_a}(a)^{-1}\omega_{\pi}(a)^3\chi(a)^6|a|^{3\left(s-\frac{1}{2}\right)}\gamma(s,\pi \otimes \chi,\mathrm{Sym}^2,\psi).
\]
\end{proposition}

\begin{proof}
Let $W^a(g):=W(t(a,1)g)$ and $W^a_{\theta^{\psi}}(\mathfrak{s}(g)):=W_{\theta^{\psi}}(\mathfrak{s}(t(a,1))\mathfrak{s}(g))$. We can check that $W^a$ is in $\mathcal{W}(\pi,\psi_a)$ and $W^a_{\theta^{\psi}}$ belongs to $\mathcal{W}(\theta^{\psi_a},\psi_a^{-1})$ (cf. \cite{Gelbart-PS}*{Lemma 5.1.3}). For $W^a \in \mathcal{W}(\pi,\psi_a)$ and $W^a_{\theta^{\psi}} \in \mathcal{W}(\theta^{\psi_a},\psi_a^{-1})$ and $f_{2s-1} \in I(2s-1,\omega^{-1}_{\pi \otimes \chi})$, we have
\[
\begin{split}
 &I(W^a,W^a_{\theta^{\psi}},f_{2s-1},\chi)\\
& =\int_{Z^2N \backslash GL_2} W(t(a,1)g)W_{\theta^{\psi}}(\mathfrak{s}(t(a,1))\mathfrak{s}(g)) f_{2s-1}(\mathfrak{s}(t(a^{-1},1))\mathfrak{s}(t(a,1))\mathfrak{s}(g)) \chi(\mathrm{det}(g)) dg.
\end{split}
\]
By changing variables $t(a,1)g \mapsto g$, we get
\[
   I(W^a,W^a_{\theta^{\psi}},f_{2s-1},\chi)=\chi(a)^{-1}|a|^{-\frac{2s+1}{4}}\int_{Z^2N \backslash GL_2} W(g)W_{\theta^{\psi}}(\mathfrak{s}(g)) f_{2s-1}(\mathfrak{s}(g))\chi(\mathrm{det}(g)) dg.
\]

\par
The dual side needs to be treated more carefully. For $W^a \in \mathcal{W}(\pi,\psi_a)$, $W^a_{\theta^{\psi}} \in \mathcal{W}(\theta^{\psi_a},\psi_a^{-1})$ and $f_{2s-1} \in I(2s-1,\omega^{-1}_{\pi \otimes \chi})$, we break $t(a,1)$ into $t(a,a)t(1,a^{-1})$ and $I_2$ into $t(1,a)t(1,a^{-1})$. Then
\[
\begin{split}
  &I(W^a,W^a_{\theta^{\psi}}, \hat{N}(s,\omega^{-1}_{\pi \otimes \chi},\psi_a)f_{2s-1},\chi)\\
  &=\omega_{\pi}(a)\mu_{\psi}(a)^{-1}(a,a^{-1})_F\int_{Z^2N \backslash GL_2} W(t(1,a^{-1})g)W_{\theta^{\psi}}(\mathfrak{s}(t(1,a^{-1}))\mathfrak{s}(g))\\
  &\quad\times  \hat{N}(s,\omega^{-1}_{\pi \otimes \chi},\psi_a)f_{2s-1}(\mathfrak{s}(t(1,a))\mathfrak{s}(t(1,a^{-1}))\mathfrak{s}(g))\chi(a)\chi(\mathrm{det}(t(1,a^{-1})g))dg.
\end{split}
\]
Changing an additive character implies a change to measures as the measure are chosen relative to the additive character. We define the fixed measure $dx$ with respect to $\psi$ as in \S \ref{sec5:1} and we denote by $d_ax$ the measure adapted to $\psi_a$. Then $d_ax=|a|^{\frac{1}{2}}dx$. Let $M_a(2s-1,\omega^{-1}_{\pi \otimes \chi})f_{2s-1}$ denote the application of the intertwining operator to $f_{2s-1}$ given by
\[
 M_a(2s-1,\omega^{-1}_{\pi \otimes \chi})f_{2s-1}(\mathfrak{s}(g))=\int_{F} f_{2s-1}(\mathfrak{s}(w_2)\mathfrak{s}(n(x))\mathfrak{s}(g)) d_ax.
\]
Using $ \gamma(2s-1,\omega_{\pi}^2\chi^4,\psi_a)=|a|^{2s-\frac{3}{2}}\omega_{\pi}^2(a)\chi^4(a)\gamma(2s-1,\omega_{\pi}^2\chi^4,\psi)$ (cf. \cite[(3.2.3)]{Tate}), $\hat{N}(2s-1,\omega^{-1}_{\pi \otimes \chi},\psi_a)$ can be written in terms of the one for $\psi$ as
\[
\begin{split}
  &\hat{N}(2s-1,\omega^{-1}_{\pi \otimes \chi},\psi_a)f_{2s-1}(\mathfrak{s}(t(1,a)))\\
  &=|a|^{2s-\frac{3}{2}}\omega_{\pi}^2(a)\chi^4(a)\gamma(2s-1,\omega_{\pi}^2\chi^4,\psi)M_a(2s-1,\omega^{-1}_{\pi \otimes \chi})f_{2s-1}(\mathfrak{s}(t(1,a)))\\
  &=|a|^{2s-\frac{3}{2}}\omega_{\pi}^2(a)\chi^4(a)\gamma(2s-1,\omega_{\pi}^2\chi^4,\psi)|a|^{\frac{2s-3}{4}+\frac{1}{2}}M(2s-1,\omega^{-1}_{\pi \otimes \chi})f_{2s-1}.
  \end{split}
\]
By making the change of variables $t(1,a^{-1})g \mapsto g$, we achieve that
\[
\begin{split}
  &I(W^a,W^a_{\theta^{\psi}}, \hat{N}(s,\omega^{-1}_{\pi \otimes \chi},\psi_a)f_{2s-1},\chi)\\
  &=\mu_{\psi_a}(a)^{-1}|a|^{\frac{5}{2}s-\frac{7}{4}}\omega^3_{\pi}(a)\chi^5(a)\int_{Z^2N \backslash GL_2} W(g)W_{\theta^{\psi}}(\mathfrak{s}(g)) \hat{N}(2s-1,\omega^{-1}_{\pi \otimes \chi},\psi)
   f_{2s-1}(\mathfrak{s}(g))\chi(\mathrm{det}(g)) dg
  \end{split}
  \]
  as desired.
\end{proof}

Henceforth it suffices to prove the stability of $\gamma$-factor with $\psi$ of conductor $\mathcal{O}$. In contrast to $\gamma$-factors in the Langlands-Shahidi method \cite[Theorem 5.1-(iv)]{Ganapathy-Lomeli}, the auxiliary factor $\mu_{\psi_a}(a)^{-1}$ appears due to the presence of the exceptional representation $\theta^{\psi}$.

\subsection{Proof of stability}
Before we launch our computation, let us fix our measures. Because $N(F),\overline{N}(F) \simeq F$ and $A(F) \simeq F^{\times}$, we identify the measure on groups $N(F)$ and $\overline{N}(F)$ with an additive measure $dx$ on $F$ and the measure on $A(F)$ with the multiplicative measure $d^{\times} a$ on $F^{\times}$. We normalize our additive measure so that $\mathrm{vol}(\mathcal{O},dx)=1$. We take the multiplicative measure $d^{\times} a$ to be $d^{\times} a=da/|a|$. Under this normalization, $d^{\times}a$ assigns $\mathcal{O}^{\times}$ to $(q-1)/q$ and $\mathfrak{p}^m$ to $q^{-m}$.

\par
Our immediate goal is to construct explicit Howe vectors $W_m \in \mathcal{W}(\pi,\psi)$, $W_{\theta^{\psi},m} \in \mathcal{W}(\theta^{\psi},\psi^{-1})$, a section $f^i_{2s-1} \in I(2s-1,\omega^{-1}_{\pi \otimes \chi})$, and a highly ramified character $\chi$ such that $I(W_m,W_{\theta^{\psi},m},f_{2s-1}^i,\chi)$ is a constant.

\begin{proposition}
\label{constant}
Let $\pi$ be an irreducible admissible representation of $GL_2$. For every $i \geq m$ and $m \geq \max\{L, \mathfrak{f}(\chi) \}$, we have
\begin{equation}
\label{eq-constant}
 I(W_m,W_{\theta^{\psi},m},f_{2s-1}^i,\chi)=q^{-3i-m}.
\end{equation}
\end{proposition}

\begin{proof}
We compute $I(W_m,W_{\theta^{\psi},m},f_{2s-1}^i,\chi)$ on the open dense subset $NZ^2 \backslash NT\overline{N}$ of $NZ^2 \backslash GL_2$. The decomposition of the Haar measure $dg=\delta_B(a)^{-1}dnd^{\times}zd^{\times}ad\overline{n}$ leads us to get
\[
\begin{split}
 I(W_m,W_{\theta^{\psi},m},f_{2s-1}^i,\chi)=&\int_{Z^2 \backslash Z} \int_{F^{\times}} \int_{F} W_m(t(z,z)t(a,1)\overline{n}(x))W_{\theta^{\psi},m}(\mathfrak{s}[t(z,z)t(a,1)\overline{n}(x)]) \\
 & \times f_{2s-1}^i(\mathfrak{s}[t(z,z)t(a,1)\overline{n}(x)]) \chi(z^2a)|a|^{-1} dx d^{\times}a d^{\times}z.
\end{split}
\]
Now we would like to write $\mathfrak{s}[t(z,z)t(a,1)\overline{n}(x)]=\mathfrak{s}(t(z,z)) \mathfrak{s}(t(a,1))\mathfrak{s}(\overline{n}(x))$. This equality does not hold in general but the fact that both $W_{\theta^{\psi},m}$ and $f_{2s-1}^i$ are genuine make it possible to do this maneuver. This leads us to exploit the property of $f^i_{2s-1}$ and we obtain
\[
  I(W_m,W_{\theta^{\psi},m},f_{2s-1}^i,\chi)=\int_{F^{\times}} \int_{\mathfrak{p}^{3i}} W_m(t(a,1)\overline{n}(x)) W_{\theta^{\psi},m}(\mathfrak{s}(t(a,1))\mathfrak{s}(\overline{n}(x))) \chi(a) |a|^{\frac{2s+1}{4}-1} dx d^{\times}a.
\]
Since $i \geq m$, we have $\mathfrak{p}^{3i} \subseteq \mathfrak{p}^{3m}$. In the light of Lemma \ref{Howe} and \cite[Lemma 5.1]{Zhang} together, we obtain $W_m(t(a,1)\overline{n}(x))=W_m(t(a,1))$ and $W_{\theta^{\psi},m}(\mathfrak{s}(t(a,1))\mathfrak{s}(\overline{n}(x)))=W_{\theta^{\psi},m}(\mathfrak{s}(t(a,1)))$. It follows from Lemma \ref{excep-a-support} joined with \cite[Lemma 5.2]{Zhang} that
\[
\begin{split}
  I(W_m,W_{\theta^{\psi},m},f_{2s-1}^i,\chi)&=q^{-3i}\int_{1+\mathfrak{p}^m} W_m(t(a,1))W_{\theta^{\psi},m}(t(a,1)) \chi(a) |a|^{\frac{2s+1}{4}-1} dx d^{\times}a\\
  &=q^{-3i}\int_{1+\mathfrak{p}^m} \chi(a) d^{\times} a.
\end{split}
\]
Then the assumption $m \geq \max\{L, \mathfrak{f}(\chi)\}$ gives rise to $I(W_m,W_{\theta^{\psi},m},f_{2s-1}^i,\chi)=q^{-3i-m}$, which concludes the proof.
\end{proof}

We express the difference of Rankin-Selberg integrals in terms of the Mellin transform of a product of certain Bessel functions. In the case of $GL_2(F) \times GL_2(F)$, \eqref{dualexpression} is basically the content of a crucial lemma of Soudry \cite[Lemma 4.5]{Soudry}.

\begin{proposition}[The Mellin transform]
Let $\pi$ and $\sigma$ be irreducible admissible representations of $GL_2$ having the same central character $\omega=\omega_{\pi}=\omega_{\sigma}$. We fix $W^1 \in \mathcal{W}(\pi,\psi)$ and $W^2 \in \mathcal{W}(\sigma,\psi)$ as above, and form Howe Whittaker functions $W^1_m$ and $W^2_m$, respectively.
 For $m \geq 6L$ and $i \geq \max\{m+1, I(N_m,\omega_{\pi\otimes\chi})\}$, we have
\begin{equation}
\label{dualexpression}
\begin{split}
  &I(W^1_m,W_{\theta^{\psi},m},M(2s-1,\omega^{-1}_{\pi \otimes \chi})f^i_{2s-1},\chi)-I(W^2_m,W_{\theta^{\psi},m},M(2s-1,\omega^{-1}_{\sigma \otimes \chi})f^i_{2s-1},\chi)\\
  &=q^{-3i-m+3L} \int_{F^{\times}} [W^1_{3L}(t(a,1)w_2)-W^2_{3L}(t(a,1)w_2)] j_{\theta^{\psi}}(a) \omega(a)^{-1} \mu_{\psi}(a^{-1})^{-1} \chi^{-1}(a)|a|^{-\frac{2s+1}{4}}d^{\times}a.
\end{split}
\end{equation}
In particular the integral can be assumed to range over the compact set $\mathfrak{p}^{-9L}$.
\end{proposition}

\begin{proof}
We take the difference and compute the dual side of the functional equation on the dense open subset $NZ^2 \backslash NTw_2N$. Then we have
\[
\begin{split}
 &I(W^1_m,W_{\theta^{\psi},m},M(2s-1,\omega^{-1}_{\pi \otimes \chi})f^i_{2s-1},\chi)-I(W^2_m,W_{\theta^{\psi},m},M(2s-1,\omega^{-1}_{\sigma \otimes \chi})f^i_{2s-1},\chi)\\
 &=\int_{Z^2 \backslash Z} \int_{F^{\times}} \int_{F} [W^1_m(t(z,z)t(1,a)w_2n(x))-W^2_m(t(z,z)t(1,a)w_2n(x))] 
 \\
  &\quad \times W_{\theta^{\psi},m}(\mathfrak{s}[t(z,z)t(1,a)w_2n(x)]) \widetilde{f}^i_{1-2s}(\mathfrak{s}[t(z,z)t(1,a)w_2n(x)]) \chi(z^2a)|a|dxd^{\times}ad^{\times}z.
\end{split}
\]
As we have seen before, we may write $\mathfrak{s}[t(z,z)t(1,a)w_2n(x)]=\mathfrak{s}(t(z,z))\mathfrak{s}(t(1,a))\mathfrak{s}(w_2)\mathfrak{s}(n(x))$, namely, 
the $``$genuineness" of $W_{\theta^{\psi},m}$ and $\widetilde{f}^i_{1-2s}$ eliminates the discrepancy of those expressions. 
By \cite[Proposition 5.3]{Zhang}, we obtain
\[
  W^1_m(t(z,z)t(1,a)w_2n(x))=W^2_m(t(z,z)t(1,a)w_2n(x))
\]
for all $n(x) \in N-N_m$. Therefore
\[
\begin{split}
 &I(W^1_m,W_{\theta^{\psi},m},M(2s-1,\omega^{-1}_{\pi \otimes \chi})f^i_{2s-1},\chi)-I(W^2_m,W_{\theta^{\psi},m},M(2s-1,\omega^{-1}_{\sigma \otimes \chi})f^i_{2s-1},\chi)\\
 &=\int_{Z^2 \backslash Z} \int_{F^{\times}} \int_{\mathfrak{p}^{-m}} [W^1_m(t(z,z)t(1,a)w_2n(x))-W^2_m(t(z,z)t(1,a)w_2n(x))] 
 \\
  &\quad  \times W_{\theta^{\psi},m}(\mathfrak{s}(t(z,z))\mathfrak{s}(t(1,a))\mathfrak{s}(w_2)\mathfrak{s}(n(x)))\\
  &\quad \times  \widetilde{f}^i_{1-2s}(\mathfrak{s}(t(z,z))\mathfrak{s}(t(1,a))\mathfrak{s}(w_2)\mathfrak{s}(n(x))) \chi(z^2a)|a|dxd^{\times}ad^{\times}z.\\
\end{split}
\]
We concentrate on the support of $t(z,z)$. For the moment we take Lemma \ref{z-support} below for granted.

\begin{lemma}
\label{z-support}
If $i \geq m+1$ and $n(x) \in \mathfrak{p}^{-m}$, then
\[
 \widetilde{f}^i_{1-2s}(\mathfrak{s}(t(z,z))\mathfrak{s}(t(1,a))\mathfrak{s}(w_2)\mathfrak{s}(n(x)))=
 \begin{cases}
 \omega^{-1}_{\pi \otimes \chi}(z) |a|^{\frac{2s-3}{4}} \widetilde{f}^i_{1-2s}(\mathfrak{s}(w_2)\mathfrak{s}(n(x))) & \quad \text{if} \quad t(z,z) \in Z^2\\
 0 & \quad \text{otherwise}.
 \end{cases}
\]
\end{lemma}

Let us resume our argument. Lemma \ref{z-support} aligned with Proposition \ref{eval-intertwining} implies that
\[
\begin{split}
 &I(W^1_m,W_{\theta^{\psi},m},M(2s-1,\omega^{-1}_{\pi \otimes \chi})f^i_{2s-1},\chi)-I(W^2_m,W_{\theta^{\psi},m},M(2s-1,\omega^{-1}_{\sigma \otimes \chi})f^i_{2s-1},\chi)\\
 &=q^{-3i}\int_{F^{\times}} \int_{\mathfrak{p}^{-m}} [W^1_m(t(1,a)w_2n(x))-W^2_m(t(1,a)w_2n(x))] W_{\theta^{\psi},m}(\mathfrak{s}(t(1,a))\mathfrak{s}(w_2)\mathfrak{s}(n(x))) \\
  &\quad  \times  \chi(a)|a|^{\frac{2s-3}{4}+1}dxd^{\times}a.
  \\
 \end{split}
\]
We see from Lemma \ref{Howe} and \cite[Lemma 5.1]{Zhang} that $W_m^i(t(1,a)w_2n(x))=\psi(x)W_m^i(t(1,a)w_2)$ and
$W_{\theta^{\psi},m}(\mathfrak{s}(t(1,a))\mathfrak{s}(w_2)\mathfrak{s}(n(x)))=\psi(x)^{-1}W_{\theta^{\psi},m}(\mathfrak{s}(t(1,a))\mathfrak{s}(w_2))$ for $i=1,2$ and all $x \in \mathfrak{p}^{-m}$. The integral becomes
\[
\begin{split}
&I(W^1_m,W_{\theta^{\psi},m},M(2s-1,\omega^{-1}_{\pi \otimes \chi})f^i_{2s-1},\chi)-I(W^2_m,W_{\theta^{\psi},m},M(2s-1,\omega^{-1}_{\sigma \otimes \chi})f^i_{2s-1},\chi)\\
 &=q^{-3i+m} \int_{F^{\times}} [W^1_m(t(1,a)w_2)-W^2_m(t(1,a)w_2)]  W_{\theta^{\psi},m}(\mathfrak{s}(t(1,a))\mathfrak{s}(w_2)) \chi(a)|a|^{\frac{2s+1}{4}}d^{\times}a.
\end{split}
\]
We concern with removing the dependence of $m$. To this end, we know from \cite[Lemma 5.1, Proposition 5.3]{Zhang} that
\begin{equation}
\label{m-removed}
\begin{split}
 &W^1_m(t(1,a)w_2)-W^2_m(t(1,a)w_2)\\
 &=\frac{1}{\mathrm{vol}(N_m)} \int_{N_{3L}} [W^1_{3L}(t(1,a)w_2n)-W^2_{3L}(t(1,a)w_2n)] \psi^{-1}(n)dn\\
 &=\frac{\mathrm{vol}(N_{3L})}{\mathrm{vol}(N_m)} [W^1_{3L}(t(1,a)w_2)-W^2_{3L}(t(1,a)w_2)].
 \end{split}
\end{equation}
We take $m$ to be $m \geq 6L$. Appealing to Proposition \ref{Bessel}-\ref{Bessel-2}, we relate the partial Bessel function with the full Bessel function by
\begin{equation}
\label{Bessel-Partial}
  \mathrm{vol}(N_m)W_{\theta^{\psi},m}(\mathfrak{s}(t(a,1))\mathfrak{s}(w_2))=j_{\theta^{\psi}}(a) \quad \text{for all} \quad |a| \leq q^{2m}.
\end{equation}
Applying the change of variables $a \mapsto a^{-1}$ and then putting \eqref{m-removed} and \eqref{Bessel-Partial} together, we have
\[
\begin{split}
  &I(W^1_m,W_{\theta^{\psi},m},M(2s-1,\omega^{-1}_{\pi \otimes \chi})f^i_{2s-1},\chi)-I(W^2_m,W_{\theta^{\psi},m},M(2s-1,\omega^{-1}_{\sigma \otimes \chi})f^i_{2s-1},\chi)\\
  &=q^{-3i-m+3L} \int_{F^{\times}} [W^1_{3L}(t(a,1)w_2)-W^2_{3L}(t(a,1)w_2)] j_{\theta^{\psi}}(a) \omega(a)^{-1} \mu_{\psi}(a^{-1})^{-1} \chi^{-1}(a)|a|^{-\frac{2s+1}{4}}d^{\times}a.
\end{split}
\]
Here $a$ is in fact taken over the compact set $\mathfrak{p}^{-9L}$ utilizing \cite[Lemma 5.2]{Zhang}.
\end{proof}

It remains to show Lemma \ref{z-support}.

\begin{proof}[Proof of Lemma \ref{z-support}]
We compute the intertwining operator $\widetilde{f}^i_{1-2s}$ by brutal force. Since
\[
\widetilde{f}^i_{1-2s}(\mathfrak{s}(t(z,z))\mathfrak{s}(t(1,a))\mathfrak{s}(w_2)\mathfrak{s}(n(x)))
   =\int_{F} f^i_{2s-1}(\mathfrak{s}(w_2)\mathfrak{s}(n(y)) \mathfrak{s}(t(z,z))\mathfrak{s}(t(1,a))\mathfrak{s}(w_2)\mathfrak{s}(n(x)) dy,
\]
the support of $\widetilde{f}^i_{1-2s}(\mathfrak{s}(t(z,z))\mathfrak{s}(t(1,a))\mathfrak{s}(w_2)\mathfrak{s}(n(x)))$ boils down to the support of 
\[
  f^i_{2s-1}(\mathfrak{s}(t(z,z))\mathfrak{s}(t(a,1)) \mathfrak{s}(w_2)\mathfrak{s}(n(y)) \mathfrak{s}(w_2)\mathfrak{s}(n(x))).
\]
We assume $t(z,z)t(a,1)\overline{n}(y)n(x)=\begin{pmatrix} \alpha\beta & \upsilon \\ & \alpha  \end{pmatrix}\overline{n}(\overline{y})$ for $\alpha \in (F^{\times})^2$, $x \in \mathfrak{p}^{-m}$ and $\overline{y} \in \mathfrak{p}^{3i}$. Then
\[
  \begin{pmatrix} za & zax \\ zy & zxy+z \end{pmatrix}=\begin{pmatrix} \alpha\beta+\delta\overline{y} & \upsilon \\ \alpha \overline{y} & \alpha \end{pmatrix}.
\]
Focusing on the bottom entries $zy=\alpha\overline{y}$ and $zxy+z=\alpha$, we conclude $z=\alpha(1-x\overline{y})$. According to the surjective map  \cite[Lemma 3.3]{ChaiZhang}, $1-x\overline{y}$, being an element of $1+\mathfrak{p}^{2m+3}$, is in fact a square element. Therefore $z$ belongs to $(F^{\times})^2$.
\end{proof}

We are now in the position to present our main result.

\begin{theorem}
Let $\pi$ and $\sigma$ be irreducible admissible representations of $GL_2$ having the same central character. Then for every sufficiently highly ramified character $\chi$ of $F^{\times}$ we have
\[
\gamma(s,\pi \otimes \chi,\mathrm{Sym}^2,\psi)=\gamma(s,\sigma \otimes \chi,\mathrm{Sym}^2,\psi).
\]
\end{theorem}

\begin{proof}
We may assume that $m \geq \max \{ \mathfrak{f}(\chi), 6L \}$. We enlarge $i$ so that $i \geq \max \{m+1,I(N_m,\omega_{\pi\otimes\chi}) \}$. Substituting \eqref{eq-constant} and \eqref{dualexpression} in the functional equation \eqref{unnoraml-intertwining}, then simply implies:
\begin{equation}
\label{gamma-difference}
\begin{split}
 & \Gamma(s,\pi \otimes \chi,\mathrm{Sym}^2,\psi)-\Gamma(s,\sigma \otimes \chi,\mathrm{Sym}^2,\psi)\\
 &=q^{3L} \int_{\mathfrak{p}^{-9L}} [W^1_{3L}(t(a,1)w_2)-W^2_{3L}(t(a,1)w_2)] j_{\theta^{\psi}}(a) \omega(a)^{-1}\mu_{\psi}(a^{-1})^{-1}  \chi^{-1}(a)|a|^{-\frac{2s+1}{4}}d^{\times}a.
 \end{split}
\end{equation}
To complete the proof, we must now observe from Proposition \ref{Bessel}-\ref{Bessel-3} that the integrand
\[
 a \mapsto [W^1_{3L}(t(a,1)w_2)-W^2_{3L}(t(a,1)w_2)] j_{\theta^{\psi}}(a) \omega(a)^{-1}\mu_{\psi}(a^{-1})^{-1}  |a|^{-\frac{2s+1}{4}}
\]
is right invariant under $1+\mathfrak{p}^{6L}$ for $a \in \mathfrak{p}^{-9L}$. Taking $\chi$ a sufficiently highly ramified character, the integral in \eqref{gamma-difference} vanishes. Since Tate's gamma factor $\gamma(2s-1,\omega^2_{\pi \otimes \chi},\psi)$ stabilizes under highly ramified twist, we conclude from Proposition \ref{division} that
\[
 \gamma(s,\pi \otimes \chi,\mathrm{Sym}^2,\psi)=\gamma(s,\sigma \otimes \chi,\mathrm{Sym}^2,\psi)
\]
as expected.
\end{proof}

 We illustrate an important consequence of this that we highlight below. We recover the result of Gelbart and Jacquet \cite[(6.4)]{Gelbart-Jacquet}.

\begin{theorem}
Let $\pi$ and $\sigma$ be irreducible admissible representations of $GL_2$ having the same central character. Then for every sufficiently highly ramified character $\chi$ of $F^{\times}$ we have
\[
   L(s,\pi \otimes \chi,\mathrm{Sym}^2)=L(s,\sigma \otimes \chi,\mathrm{Sym}^2)= 1.
\]
Moreover $\varepsilon(s,\pi \otimes \chi,\mathrm{Sym}^2,\psi)= \varepsilon(s,\sigma \otimes \chi,\mathrm{Sym}^2,\psi)$.
\end{theorem}

\begin{proof}
 In virtue of Yamana \cite[Lemma 3.15]{Ya17} concerning the pole of the symmetric square $L$-function, this can be proved exactly as in Jacquet and Shalika \cite[Proposition 5.1]{JS}. 
\end{proof}

\begin{acknowledgments}
The author wishes to thank J. Cogdell for pointing out the unfinished project in A. Kable \cite{Kable01} and for encouraging us to write this article. The second part  ``stability of $\gamma$-factors" grows out of several discussions with M. Krishnamurthy. We wish to express our gratitude to M. Krishnamurthy for keen insights and for fruitful mathematical communications. We also thank every members of the number theory and the representation theory group at the University of Iowa for vibrating environments when this paper was written. Finally, we would like to convey our sincere appreciation to the anonymous referee for invaluable suggestions which improve exposition and organization of our paper. 
\end{acknowledgments}

\bibliographystyle{amsplain}

\begin{bibdiv}
\begin{biblist}

\bib{BLS}{article}{
   author={Banks, William D.},
   author={Levy, Jason},
   author={Sepanski, Mark R.},
   title={Block-compatible metaplectic cocycles},
   journal={J. Reine Angew. Math.},
   volume={507},
   date={1999},
   pages={131--163},
  }
  
     
  \bib{BernsteinZelevinsky}{article}{
   author={Bernstein, I. N.},
   author={Zelevinsky, A. V.},
   title={Induced representations of reductive ${\germ p}$-adic groups. I},
   journal={Ann. Sci. \'{E}cole Norm. Sup. (4)},
   volume={10},
   date={1977},
   number={4},
   pages={441--472},
  }
  
  \bib{BuGi}{article}{
   author={Bump, Daniel},
   author={Ginzburg, David},
   title={Symmetric square $L$-functions on ${\rm GL}(r)$},
   journal={Ann. of Math. (2)},
   volume={136},
   date={1992},
   number={1},
   pages={137--205},
   }
   
   \bib{Chai}{article}{
   author={Chai, Jingsong},
   title={Some results on archimedean Rankin-Selberg integrals},
   journal={Pacific J. Math.},
   volume={273},
   date={2015},
   number={2},
   pages={277--305},
  }
   
   \bib{ChaiZhang}{article}{
   author={Chai, Jingsong},
   author={Zhang, Qing},
   title={A strong multiplicity one theorem for $\rm SL_2$},
   journal={Pacific J. Math.},
   volume={285},
   date={2016},
   number={2},
   pages={345--374},
   }

\bib{Chen}{article}{
   author={Chen, Shih-Yu},
   title={Gamma factors for the Asai cube representation},
   journal={to appear in Mathematische Zeitschrift},
   note={https://arxiv.org/abs/arXiv:1904.07844},
    }
     
   \bib{Cogdell-PS}{article}{
   author={Cogdell, J. W.},
   author={Piatetski-Shapiro, I. I.},
   title={Derivatives and L-functions for $GL_n$},
   conference={
      title={Representation theory, number theory, and invariant theory},
   },
   book={
      series={Progr. Math.},
      volume={323},
      publisher={Birkh\"{a}user/Springer, Cham},
   },
   date={2017},
   pages={115--173},
}

\bib{CoShTs}{article}{
   author={Cogdell, J. W.},
   author={Shahidi, F.},
   author={Tsai, T.-L.},
   title={Local Langlands correspondence for ${\rm GL}_n$ and the exterior
   and symmetric square $\varepsilon$-factors},
   journal={Duke Math. J.},
   volume={166},
   date={2017},
   number={11},
   pages={2053--2132},
  }
  

\bib{Ganapathy-Lomeli}{article}{
   author={Ganapathy, Radhika},
   author={Lomel\'{\i}, Luis},
   title={On twisted exterior and symmetric square $\gamma$-factors},
   language={English, with English and French summaries},
   journal={Ann. Inst. Fourier (Grenoble)},
   volume={65},
   date={2015},
   number={3},
   pages={1105--1132},
  }
  
  \bib{GaoShahidiSzpruch}{article}{
   author={Gao, Fan},
   author={Shahidi, Freydoon},
   author={Szpruch, Dani},
   title={Local coefficients and gamma factors for principal series of covering groups},
   journal={to appear in Mem. Amer. Math. Soc.},
   note={https://arxiv.org/abs/arXiv:1902.02686v3},
    }
    
    \bib{Gelbart-Jacquet}{article}{
   author={Gelbart, Stephen},
   author={Jacquet, Herv\'{e}},
   title={A relation between automorphic representations of ${\rm GL}(2)$
   and ${\rm GL}(3)$},
   journal={Ann. Sci. \'{E}cole Norm. Sup. (4)},
   volume={11},
   date={1978},
   number={4},
  }
   
   
      
 \bib{Gelbart-PS}{article}{
   author={Gelbart, Stephen},
   author={Piatetski-Shapiro, I. I.},
   title={Distinguished representations and modular forms of half-integral
   weight},
   journal={Invent. Math.},
   volume={59},
   date={1980},
   number={2},
   pages={145--188},
  }
  
  
  
  \bib{HKS}{article}{
   author={Harris, Michael},
   author={Kudla, Stephen S.},
   author={Sweet, William J.},
   title={Theta dichotomy for unitary groups},
   journal={J. Amer. Math. Soc.},
   volume={9},
   date={1996},
   number={4},
   pages={941--1004},
  }
  
  \bib{Henniart}{article}{
   author={Henniart, Guy},
   title={Correspondance de Langlands et fonctions $L$ des carr\'{e}s ext\'{e}rieur
   et sym\'{e}trique},
   language={French},
   journal={Int. Math. Res. Not. IMRN},
   date={2010},
   number={4},
   pages={633--673},
  }
      
    \bib{Howe}{article}{
   author={Howe, Roger},
   title={Classification of Irreducible Representations of $GL_2(F)$},
   journal={preprint, I.H.E.S., Bures-sur-Yvette, France},
   date={1978},
   }
   
   \bib{JS}{article}{
   author={Jacquet, Herv\'{e}},
   author={Shalika, Joseph},
   title={A lemma on highly ramified $\epsilon$-factors},
   journal={Math. Ann.},
   volume={271},
   date={1985},
   number={3},
   pages={319--332},
   }
   
   \bib{JPSS}{article}{
   author={Jacquet, H.},
   author={Piatetskii-Shapiro, I. I.},
   author={Shalika, J. A.},
   title={Rankin-Selberg convolutions},
   journal={Amer. J. Math.},
   volume={105},
   date={1983},
   number={2},
   pages={367--464},
  }
  
  \bib{JO-3}{article}{
   author={Jo, Yeongseong},
   title={Derivatives and exceptional poles of the local exterior square
   $L$-function for $GL_m$},
   journal={Math. Z.},
   volume={294},
   date={2020},
   number={3-4},
   pages={1687--1725},
  }
 
  \bib{JO20-2}{article}{
   author={Jo, Yeongseong},
   title={Factorization of the local exterior square $L$-function of $GL_m$},
   journal={Manuscripta Math.},
   volume={162},
   date={2020},
   number={3-4},
   pages={493--536},
   }  
 
\bib{JO20}{article}{
   author={Jo, Yeongseong},
   title={Rankin$-$Selberg $L$-functions via good sections},
   journal={Forum Math.},
   volume={32},
   date={2020},
   number={4},
   pages={1039--1074},
   }
    
  \bib{Kable99}{article}{
   author={Kable, Anthony C.},
   title={The main involutions of the metaplectic group},
   journal={Proc. Amer. Math. Soc.},
   volume={127},
   date={1999},
   number={4},
   pages={955--962},
  }
  
  \bib{Kable01}{article}{
   author={Kable, Anthony C.},
   title={The tensor product of exceptional representations on the general
   linear group},
   language={English, with English and French summaries},
   journal={Ann. Sci. \'{E}cole Norm. Sup. (4)},
   volume={34},
   date={2001},
   number={5},
   pages={741--769},
  }
  
  \bib{Kaplan13}{article}{
   author={Kaplan, Eyal},
   title={On the gcd of local Rankin-Selberg integrals for even orthogonal
   groups},
   journal={Compos. Math.},
   volume={149},
   date={2013},
   number={4},
   pages={587--636},
  }

\bib{Kaplan17}{article}{
   author={Kaplan, Eyal},
   title={The characterization of theta-distinguished representations of
   ${\rm GL}(n)$},
   journal={Israel J. Math.},
   volume={222},
   date={2017},
   number={2},
   pages={551--598},
  }

\bib{Kaplan}{article}{
   author={Kaplan, Eyal},
   title={The double cover of odd general spin groups, small
   representations, and applications},
   journal={J. Inst. Math. Jussieu},
   volume={16},
   date={2017},
   number={3},
   pages={609--671},
  }
  
  \bib{KaPa}{article}{
   author={Kazhdan, D. A.},
   author={Patterson, S. J.},
   title={Metaplectic forms},
   journal={Inst. Hautes \'{E}tudes Sci. Publ. Math.},
   number={59},
   date={1984},
   pages={35--142},
  }
  
 
 
  \bib{Lojasiewicz}{book}{
   author={\L ojasiewicz, Stanis\l aw},
   title={Introduction to complex analytic geometry},
   note={Translated from the Polish by Maciej Klimek},
   publisher={Birkh\"{a}user Verlag, Basel},
   date={1991},
   pages={xiv+523},
   }
  
  \bib{Matringe09}{article}{
   author={Matringe, Nadir},
   title={Conjectures about distinction and local Asai $L$-functions},
   journal={Int. Math. Res. Not. IMRN},
   date={2009},
   number={9},
   pages={1699--1741},
   }
		
  
  \bib{Matringe}{article}{
   author={Matringe, Nadir},
   title={On the local Bump-Friedberg $L$-function},
   journal={J. Reine Angew. Math.},
   volume={709},
   date={2015},
   pages={119--170},
   issn={0075-4102},
  }
  
  \bib{PP}{article}{
   author={Patterson, S. J.},
   author={Piatetski-Shapiro, I. I.},
   title={The symmetric-square $L$-function attached to a cuspidal
   automorphic representation of ${\rm GL}_3$},
   journal={Math. Ann.},
   volume={283},
   date={1989},
   number={4},
   pages={551--572},
  }
  
  \bib{Piatetski-Shapiro}{article}{
   author={Piatetski-Shapiro, I. I.},
   title={$L$-functions for ${\rm GSp}_4$},
   note={Olga Taussky-Todd: in memoriam},
   journal={Pacific J. Math.},
   date={1997},
   number={Special Issue},
   pages={259--275},
 }
  
 \bib{PSRA87}{article}{
   author={Piatetski-Shapiro, I.},
   author={Rallis, Stephen},
   title={Rankin triple $L$ functions},
   journal={Compositio Math.},
   volume={64},
   date={1987},
   number={1},
   pages={31--115},
  }
  
  \bib{Schmidt-Tran}{article}{
   author={Schmidt, Ralf},
   author={Tran, Long},
   title={Zeta integrals for ${\rm GSp}(4)$ via Bessel models},
   journal={Pacific J. Math.},
   volume={296},
   date={2018},
   number={2},
   pages={437--480},
  }
  
  \bib{Shahidi}{article}{
   author={Shahidi, Freydoon},
   title={Twisted endoscopy and reducibility of induced representations for
   $p$-adic groups},
   journal={Duke Math. J.},
   volume={66},
   date={1992},
   number={1},
   pages={1--41},
  }
  
  \bib{Shimura}{article}{
   author={Shimura, Goro},
   title={On the holomorphy of certain Dirichlet series},
   journal={Proc. London Math. Soc. (3)},
   volume={31},
   date={1975},
   number={1},
   pages={79--98},
   }
  
  \bib{Soudry}{article}{
   author={Soudry, David},
   title={The $L$ and $\gamma $ factors for generic representations of ${\rm
   GSp}(4,\,k)\times {\rm GL}(2,\,k)$ over a local non-Archimedean field
   $k$},
   journal={Duke Math. J.},
   volume={51},
   date={1984},
   number={2},
   pages={355--394},
  }

  \bib{TA14}{article}{
   author={Takeda, Shuichiro},
   title={The twisted symmetric square $L$-function of $\mathrm{GL}(r)$},
   journal={Duke Math. J.},
   volume={163},
   date={2014},
   number={1},
   pages={175--266},
   }
   
   \bib{Takeda15}{article}{
   author={Takeda, Shuichiro},
   title={On a certain metaplectic Eisenstein series and the twisted
   symmetric square $L$-function},
   journal={Math. Z.},
   volume={281},
   date={2015},
   number={1-2},
   pages={103--157},
   }
   
  
  \bib{Tate}{article}{
   author={Tate, J.},
   title={Number theoretic background},
   conference={
      title={Automorphic forms, representations and $L$-functions},
      address={Proc. Sympos. Pure Math., Oregon State Univ., Corvallis,
      Ore.},
      date={1977},
   },
   book={
      series={Proc. Sympos. Pure Math., XXXIII},
      publisher={Amer. Math. Soc., Providence, R.I.},
   },
   date={1979},
   pages={3--26},
 }
  
\bib{Wal03}{article}{
   author={Waldspurger, J.-L.},
   title={La formule de Plancherel pour les groupes $p$-adiques (d'apr\`es
   Harish-Chandra)},
   language={French, with French summary},
   journal={J. Inst. Math. Jussieu},
   volume={2},
   date={2003},
   number={2},
   pages={235--333},
  }
  
  \bib{Wang}{book}{
   author={Wang, Chian-Jen},
   title={On the existence of cuspidal distinguished representations of
   metaplectic groups},
   note={Thesis (Ph.D.)--The Ohio State University},
   publisher={ProQuest LLC, Ann Arbor, MI},
   date={2003},
   pages={99},
   }

 \bib{Weil}{book}{
   author={Weil, Andr\'{e}},
   title={Basic number theory},
   series={Classics in Mathematics},
   note={Reprint of the second (1973) edition},
   publisher={Springer-Verlag, Berlin},
   date={1995},
   pages={xviii+315},
  }
  
  
  \bib{Ya14}{article}{
   author={Yamana, Shunsuke},
   title={L-functions and theta correspondence for classical groups},
   journal={Invent. Math.},
   volume={196},
   date={2014},
   number={3},
   pages={651--732},
  }
  
  
   \bib{Ya17}{article}{
   author={Yamana, Shunsuke},
   title={Local symmetric square $L$-factors of representations of general
   linear groups},
   journal={Pacific J. Math.},
   volume={286},
   date={2017},
   number={1},
   pages={215--256},
  }

\bib{Zhang}{article}{
   author={Zhang, Qing},
   title={A local converse theorem for $\rm U(1,1)$},
   journal={Int. J. Number Theory},
   volume={13},
   date={2017},
   number={8},
   pages={1931--1981},
   }

\end{biblist}
\end{bibdiv}

\end{document}